\documentclass{amsart}
\usepackage{shortsalch, xy}
\xyoption{all}
\title{Homotopy colimits in stable representation theory.}
\date{February 2013}

\begin{document}

\begin{abstract}
We study the problem of existence and uniqueness of homotopy colimits in stable representation theory, where one typically does not have model category structures to guarantee that these homotopy colimits exist or have good properties. We get both negative results (homotopy cofibers fail to exist if there exist any objects of positive finite projective dimension!) and positive results (reasonable conditions under which homotopy colimits exist and are unique, even when model category structures fail to exist). Along the way, we obtain relative-homological-algebraic generalizations of classical theorems of Hilton-Rees and Oort. We describe some applications to Waldhausen $K$-theory and to deformation-theoretic methods in stable representation theory.
\end{abstract}

\maketitle
\tableofcontents

\section{Introduction.}

Suppose $\mathcal{C}$ is an abelian category---for examples, the category of modules over a ring, or the category of abelian sheaves on a scheme.
By {\em stable representation theory} one means the study of $\mathcal{C}$ under the equivalence relation in which one regards two maps $f,g$ in $\mathcal{C}$
as being equivalent, or ``homotopic,'' if $f-g$ factors through a projective object. One says that two objects in $\mathcal{C}$ are ``stably equivalent''
if they become isomorphic after imposing this equivalence relation on maps in $\mathcal{C}$.
Since stably equivalent objects in $\mathcal{C}$ have the same $\Ext^n_{\mathcal{C}}$ groups for all $n>0$,
stable representation theory is a natural topic of study if one wants to compute the higher $\Ext_{\mathcal{C}}$ groups for a large family of objects
(or perhaps {\em all} objects) in $\mathcal{C}$.

In this paper we consider the problem of the existence and uniqueness of {\em homotopy colimits} in stable representation theory. Specifically, if one has a 
diagram of objects in an abelian category $\mathcal{C}$, and all of the morphisms in the diagram are monomorphisms, one wants to know that replacing
an object in the diagram with a stably equivalent object will not change the colimit of the diagram, up to stable equivalence. 
Here are some reasons why one wants to do this:
\begin{itemize}
\item One wants to study and compute the stable algebraic $G$-theory associated to $\mathcal{C}$, that is, 
one wants to study {\em derived stable representation theory,} in the sense that $G_0(\mathcal{C})$ is the Grothendieck group completion of a monoid of stable equivalence classes of objects in $\mathcal{C}$, and the higher $G$-theory groups capture more subtle $K$-theoretic invariants of the stable representation theory of $\mathcal{C}$. We do some of this in our paper \cite{g-theory1}, using results from the present paper.

To construct the relevant $G$-theory, one needs the structure of a Waldhausen category on $\mathcal{C}$ 
in which
the weak equivalences are the stable equivalences. But one of the axioms required of a Waldhausen category, Waldhausen's axiom {\bf Weq 2} from \cite{MR802796},
is that, given a commutative diagram in $\mathcal{C}$
\[ \xymatrix{
X^{\prime} \ar[d] & Y^{\prime}\ar[l]\ar[r]\ar[d] & Z^{\prime}\ar[d] \\
X  & Y \ar[l]\ar[r] & Z }\]
in which the horizontal maps are cofibrations and the vertical maps
are weak equivalences, the induced map of pushouts
\[ X^{\prime}\coprod_{Y^{\prime}} Z^{\prime}\rightarrow X\coprod_Y Z \]
is a weak equivalence. In other words, {\em homotopy pushouts are well-defined
in $\mathcal{C}$.} So one must know something about well-definedness
of homotopy pushouts in order to do any $K$-theory or $G$-theory.
\item One wants to be able to make constructions in stable representation theory 
which come from {\em geometric realization of simplicial objects} 
and {\em totalization of cosimplicial objects.}
For example, topological Hochschild homology and topological Andre-Quillen 
homology occur as geometric realizations, while topological Hochschild
cohomology and topological Andre-Quillen cohomology occur as 
totalizations. Cocycles classifying deformations of modules and algebras
live in topological Hochschild cohomology and topological Andre-Quillen
cohomology groups, and if one wants to use deformation-theoretic methods
in stable representation theory, one wants to be able to form the
necessary totalizations. 

Meanwhile, topological Andre-Quillen homology
occurs as homology of cotangent complexes, which also has applications
in deformation theory, while under certain conditions, topological Hochschild 
homology receives a trace map from algebraic $K$-theory which is quite useful 
for making $K$-theory computations. So one wants to be able to 
form the necessary geometric realizations to construct these objects.

Geometric realizations are particular kinds of homotopy colimits and we study
their existence and uniqueness in this paper. On the other hand, totalizations are homotopy
limits, not colimits, so we put off their question of their existence and
uniqueness in stable representation theory for a later paper.
\end{itemize}

Existence and uniqueness of homotopy colimits is well-understood in the context
of a model category, but {\em abelian categories frequently do not admit the
structure of a model category} in which the weak equivalences are
the stable equivalences and the cofibrations are the monomorphisms! So one cannot rely on general model-category-theoretic
methods. 

In fact, we get some negative results, which preclude the existence of
such a model category structure (or even a Waldhausen category structure)
on an abelian category under surprising conditions:
a special case of our Cor.~\ref{main nonexistence corollary on htpy pushouts}
is that {\em if an abelian category $\mathcal{C}$ with enough projectives
has any objects of finite, positive projective dimension, then
homotopy cofibres fail to be well-defined in $\mathcal{C}$.}
As a consequence, if there exists a single object of projective 
dimension $\neq 0,\infty$, then $\mathcal{C}$ cannot have a model category structure
or a Waldhausen category structure with the desired cofibrations
and weak equivalences!

On the other hand, suppose that $\mathcal{C}$ has enough projectives and enough 
injectives,
and suppose that every projective object is injective. Then homotopy cofibers
(and homotopy pushouts in general) are unique up to stable equivalence; this is
a special case of our Cor.~\ref{main prop on well-definedness of homotopy pushouts}. As a consequence, $\mathcal{C}$ then satisfies Waldhausen's axiom {\bf Weq 2}.
This is substantially weaker than the assumption that $\mathcal{C}$ is quasi-Frobenius (i.e., projective objects
coincide with injective objects), which is the known condition
under which $\mathcal{C}$ admits a model category structure with the desired
properties, as in \cite{MR1650134}.

We also show that, when $\mathcal{C}$ has enough projectives and enough 
injectives, when every projective object is injective, and when every object
can be embedded appropriately into a projective object, then geometric realization of
simplicial objects is well-defined in $\mathcal{C}$; this is a special case of our
Cor.~\ref{geometric realization is well-defined}.

Throughout this paper, we work not just with abelian categories, but with
abelian categories with a specified {\em allowable class}; that is, we 
work in the context of {\em relative homological algebra}. A good treatment
of the basics of this subject is in Mac Lane's book \cite{MR1344215},
but the appendix to this paper is a self-contained introduction to the
subject, so that the reader will not have to look elsewhere for the basic
definitions. 

There are two reasons we work in the context of relative homological algebra:
\begin{itemize}
\item
One wants to study the effect of localization, i.e., change of allowable class, 
on algebraic $G$-theory. In \cite{g-theory1} we use such localization 
methods to obtain
some long exact sequences in stable algebraic $G$-theory groups, i.e.,
derived stable representation groups, of Hopf algebras.
Our results in \cite{g-theory1} require the results on well-definedness
of homotopy pushouts, in particular Cor.~\ref{existence result on weq2},
from the present paper.
\item 
Our main area of applications for these results is in the stable
representation theory of comodules over Hopf algebroids, especially those arising 
in stable homotopy theory. The $\Ext$ groups in the category of
comodules over various Hopf algebroids are the $E_2$-terms of generalized
Adams spectral sequences which are used to compute stable homotopy groups of 
various spaces and spectra, so the stable representation theory of these comodules
is quite important for topology.  If $(A,\Gamma)$ is a Hopf algebroid,
the relevant homological algebra is the one in which the
relative projective objects are the comodules which are tensored up from
$A$-modules; see Appendix 1 of \cite{MR860042} for these ideas.
Since comodules over certain Hopf algebroids are equivalent to quasicoherent modules over
certain Artin stacks, this direction is relevant to algebraic geometry as well.
\end{itemize}

We note that three essential technical tools in this paper are the 
relative-homological-algebraic generalizations of classical theorems in the theory of abelian categories: namely, our 
Lemmas~\ref{stable hom and nat trans of ext1} and
\ref{iso in ext1 means iso after direct sum with projectives} are the relative
versions of the main results of the 1961 paper \cite{MR0124377} of Hilton and
Rees, and our Prop.~\ref{generalization of oort's result}
is the relative version of the main result of the 1963 paper
\cite{MR0146233} of Oort. These generalizations are, to our knowledge, new,
but they are not difficult: one can simply mimic the proofs of Hilton-Rees and Oort, with appropriate adjustments for the more general setting.

Even if one has no interest in relative homological algebra or in abelian categories aside from categories of modules over a ring, 
our positive results still have some ``teeth'':
there is an open conjecture in pure algebra that the category of finitely-generated modules over
a ring $R$ is quasi-Frobenius if and only if every finitely-generated $R$-module embeds in a projective $R$-module. See \cite{MR1668220} for 
some discussion of this problem. This conjecture is the analogue for {\em finitely-generated} modules of the theorem of Faith and Walker (a good reference is \cite{MR6728009}), which states that the category of {\em all} $R$-modules is quasi-Frobenius if and 
only if every $R$-module embeds in a projective $R$-module. The main point of the section on homotopy pushouts and homotopy cofibres in the present paper
is that one weakening of the (relative) quasi-Frobenius condition---the condition
that every relatively projective object be relatively injective---suffices to ensure that homotopy pushouts are well-defined. 
This result is Prop.~\ref{main prop on well-definedness of homotopy pushouts}.
So if one wants to study the stable representation theory of {\em finitely-generated} $R$-modules (which is what one must do
in order for $K$-theoretic approaches like stable $G$-theory to be applicable, to avoid an Eilenberg swindle forcing all $K$-groups to be trivial!), 
then being able to embed such modules
in projectives is {\em not} known to imply the quasi-Frobenius property and hence such module categories are {\em not}
known to admit the structure of a model category---but one
still has some good properties (e.g. homotopy cofibres, and Waldhausen's axiom {\bf Weq 2}) in such categories of modules, by the results in the
present paper. 

Finally, we list the main results in this paper, for ease of reference:
\begin{itemize}
\item Cor.~\ref{main nonexistence corollary on htpy pushouts} states
that, in a relative abelian category with enough 
relative projectives and in which there exists an object
of positive, finite relative projective dimension, homotopy
pushouts (including homotopy cofibers) fail to be unique up to homotopy equivalence.
\item Cor.~\ref{nonexistence result on weq2} states that, under the same
hypotheses, such a relative abelian category does not satisfy
Waldhausen's axiom {\bf Weq 2}.
\item Prop.~\ref{main prop on well-definedness of homotopy pushouts}
states that, under mild assumptions on a relative abelian category,
if every relatively projective object is relatively injective,
then homotopy pushouts (including homotopy cofibers) {\em are} unique up to homotopy equivalence.
\item Cor.~\ref{existence result on weq2} states that,
under the same hypotheses, the relative abelian category
{\em does} satisfy Waldhausen's axiom {\bf Weq 2}.
\item Prop.~\ref{sequential colimits are well-defined}
states that, under mild assumptions on a relative abelian category,
if every relatively projective object is relatively injective,
then sequential homotopy colimits are unique up to homotopy equivalence.
\item Cor.~\ref{geometric realization is well-defined}
states that, under the same hypotheses, 
if the abelian category is AB3 and
every object embeds appropriately in a relatively projective object, 
then geometric realizations of simplicial objects
exist and are unique up to homotopy equivalence.
\end{itemize}

\section{Definitions.}

We define a category with a minimal amount of structure one wants in order to consider the notion of homotopy colimits therein:
\begin{definition} 
By a {\em weak Waldhausen category} we mean a category $\mathcal{C}$ equipped with a pair of subcategories
$cof(\mathcal{C}),we(\mathcal{C})$ of $\mathcal{C}$ such that:
\begin{itemize}
\item For each object $X$ of $\mathcal{C}$, the identity map on $X$ is in both $cof(\mathcal{C})$ and $we(\mathcal{C})$. (In other words,
$cof(\mathcal{C})$ and $we(\mathcal{C})$ are both ``lluf.'')
\end{itemize}
\end{definition}

The notion of a weak Waldhausen category is so general and so weakly-structured that one can't hope to prove much about such objects. 
Using some ideas from relative homological algebra, 
we now define a particular class of weak Waldhausen categories that we can actually prove some things about:
\begin{definition} 
By a {\em weak Waldhausen abelian category} we mean an abelian category $\mathcal{C}$ equipped with the structure of a weak Waldhausen category, and such that
there exists an allowable class $E$ in $\mathcal{C}$ with the property that:
\begin{itemize}
\item $we(\mathcal{C})$ is equal to the class of $E$-stable equivalences, and
\item $cof(\mathcal{C})$ is equal to the class of $E$-monomorphisms.
\end{itemize}
If $E$ is such an allowable class for a given weak Waldhausen abelian category $\mathcal{C}$, we will say that $E$ is {\em allowable for $\mathcal{C}$.}
\end{definition}
The appendix to this paper provides some useful classical definitions for the reader unfamiliar with relative homological algebra. In particular,
``allowable class'' is Def.~\ref{def of allowable class} and ``$E$-stable equivalence'' is Def.~\ref{def of stable equivalence}.

The following is Waldhausen's axiom {\bf Weq 2}, which we will be concerned with:
\begin{definition}
Let $\mathcal{C}$ be a weak Waldhausen category. We say that $\mathcal{C}$ {\em satisfies Waldhausen's axiom {\bf Weq 2}} 
if, for each commutative diagram
\[ \xymatrix{ 
X^{\prime}\ar[d] & Y^{\prime} \ar[l]\ar[r] \ar[d] & Z^{\prime} \ar[d] \\
X^{} & Y^{} \ar[l]\ar[r]  & Z^{}  }\]
in which the maps $Y^{\prime}\rightarrow X^{\prime}$ and $Y\rightarrow X$ are cofibrations and the vertical maps are all 
weak equivalences, then the map 
\[ X^{\prime}\coprod_{Y^{\prime}} Z^{\prime} \rightarrow X\coprod_Y Z\]
is a weak equivalence.
\end{definition}

\begin{definition}\label{def of homotopy colimit diagram}
Let $\mathcal{D}$ be a small category and let $\mathcal{C}$ be a category with a distinguished class of morphisms $cof(\mathcal{C})$ (for example,
$\mathcal{C}$ could be a weak Waldhausen abelian category). Suppose $\mathcal{C}$ has an initial object $0$.
By  a {\em $\mathcal{D}$-indexed homotopy colimit diagram in $\mathcal{C}$} we mean a functor $F: \mathcal{D}\rightarrow\mathcal{C}$ with the following properties:
\begin{itemize}
\item For each object $X$ of $\mathcal{D}$, the map $0\rightarrow F(X)$ is in $cof(\mathcal{C})$.
\item For each map $f: X\rightarrow Y$ in $\mathcal{D}$, the map $F(f): F(X)\rightarrow F(Y)$ is in $cof(\mathcal{C})$.
\end{itemize}
\end{definition}
We shall see, in Lemma~\ref{split monos are in E}, that if $\mathcal{C}$ is a weak Waldhausen abelian category with an allowable class that has sectile epics,
then the first condition (that the map $0\rightarrow F(X)$ be a cofibration)
in Def.~\ref{def of homotopy colimit diagram} is actually redundant---it is automatically satisfied. The second condition in Def.~\ref{def of homotopy colimit diagram}
is the significant one. See Def.~\ref{def of retractile monics} for the definition of ``having sectile epics.''

\begin{definition}
Let $\mathcal{D}$ be a small category and let $\mathcal{C}$ be a weak Waldhausen abelian category. Suppose $E$ is a class allowable for $\mathcal{C}$.
We say that {\em $\mathcal{D}$-indexed homotopy colimits of $E$-projectives in $\mathcal{C}$ are $E$-projective}
if, whenever $F:\mathcal{D}\rightarrow\mathcal{C}$ is a homotopy colimit diagram such that $F(X)$ is $E$-projective 
for every object $X$ of $\mathcal{D}$, the colimit $\colim F$ exists and is $E$-projective.
\end{definition}

The essential property that one wants in a homotopy colimit is that it should be homotopy-invariant. In a model category, one always knows that this is so.
But in our much, much more general situation, that of a weak Waldhausen abelian category, some homotopy colimit diagrams may fail to have homotopy-invariant colimits. 
When this is so, we say that the homotopy colimit in question {\em fails to be well-defined.} Precisely:
\begin{definition}
Let $\mathcal{D}$ be a small category and let $\mathcal{C}$ be a weak Waldhausen abelian category. Suppose
$\mathcal{C}$ has all $\mathcal{D}$-indexed colimits (e.g. $\mathcal{C}$ is an AB3, that is, co-complete, abelian category).
We say that {\em $\mathcal{D}$-indexed homotopy colimits are well-defined,} or {\em unique up to homotopy, in $\mathcal{C}$} if, 
for any pair of $\mathcal{D}$-indexed homotopy colimit diagrams $F,G: \mathcal{D}\rightarrow\mathcal{C}$ and any
map of diagrams $\phi: F\rightarrow G$ such that $\phi(X): F(X)\rightarrow G(X)$ is in $we(\mathcal{C})$ for every object $X$ of $\mathcal{D}$,
the induced map $\colim F\rightarrow \colim G$ is in $we(\mathcal{C})$.
\end{definition}

Finally, we include definitions related to lengths of $E$-projective resolutions, which we will use in the next section:
\begin{definition}
Suppose $\mathcal{C}$ is an abelian category, $E$ an allowable class in $\mathcal{C}$.
We say that a long exact sequence 
\[ \dots \stackrel{f_n}{\longrightarrow} P_{n-1} \stackrel{f_{n-1}}{\longrightarrow} \dots \stackrel{f_{1}}{\longrightarrow} P_0 
 \stackrel{f_0}{\longrightarrow} X \rightarrow 0 \]
is {\em $E$-long exact} if each short exact sequence
\[ 0\rightarrow \ker f_{i+1} \rightarrow P_i \rightarrow \coker f_{i+1} \rightarrow 0\]
is in $E$.
If each $P_i$ is an $E$-projective object, we say that the $E$-long exact sequence is an {\em $E$-projective resolution of $X$}.
\end{definition}

\begin{definition}\label{def of E-proj dim}
Suppose $\mathcal{C}$ is an abelian category, $E$ an allowable class in $\mathcal{C}$.
Suppose $n$ is a nonnegative integer.
We say that an object $X$ of $\mathcal{C}$ {\em has $E$-projective dimension $\leq n$} if there exists
an $E$-projective resolution of $X$
\[ 0 \rightarrow P_n \stackrel{f_n}{\longrightarrow} P_{n-1} \stackrel{f_{n-1}}{\longrightarrow} \dots \stackrel{f_{1}}{\longrightarrow} P_0 
 \stackrel{f_0}{\longrightarrow} X \rightarrow 0. \]

We say that $X$ {\em has $E$-projective dimension $n$} if it has $E$-projective dimension $\leq n$ but does not have
$E$-projective dimension $\leq n-1$.
\end{definition}

\section{Negative results on all homotopy colimits.}

In this section we prove that homotopy colimits in a weak Waldhausen abelian category fail to be unique up to homotopy 
unless colimits of appropriately-shaped relative projectives are themselves relatively projective. A precise statement is in 
Prop.~\ref{existence of htpy colimits implies projectivity of colimits of projectives}. 
An important application of Prop.~\ref{existence of htpy colimits implies projectivity of colimits of projectives} 
is the case of homotopy pushouts (and, in particular, homotopy cofibers) in Cor.~\ref{main nonexistence corollary on htpy pushouts}, and the question of whether
Waldhausen's axiom {\bf Weq 2} is satisfied, which we address in Cor.~\ref{nonexistence result on weq2}.

We begin by proving some lemmas that will allow us to prove both well-definedness and non-well-definedness results about homotopy colimits.
\begin{lemma}\label{easy lemma on nonvanishing of Ext^1}
Let $\mathcal{C}$ be an abelian category and let $E$ be an allowable class in $\mathcal{C}$. Suppose $\mathcal{C}$ has enough $E$-projectives.
Then an object $X$ of $\mathcal{C}$ is $E$-projective if and only if $\Ext^1_{\mathcal{C}/E}(X, Y)\cong 0$ for all objects $Y$ of $\mathcal{C}$.
\end{lemma}
\begin{proof}
Suppose $X$ is $E$-projective. Then vanishing of $\Ext^1_{\mathcal{C}/E}(X, Y)$ is classical (and easy).

Now suppose $X$ is not $E$-projective. Then there exists some $E$-epimorphism $f: A\rightarrow B$ and a map $g: X\rightarrow B$ which does not lift through $f$.
In other words, the element $g\in \hom_{\mathcal{C}}(X, B)$ is not in the image of the map 
\[ \hom_{\mathcal{C}}(X, A) \rightarrow \hom_{\mathcal{C}}(X, B).\]
But we have the exact sequence 
\[ \hom_{\mathcal{C}}(X, A) \rightarrow \hom_{\mathcal{C}}(X, B)\rightarrow \Ext^1_{\mathcal{C}/E}(X, \ker f) \]
and so $g$ must have nonzero image in $\Ext^1_{\mathcal{C}/E}(X, \ker f)$. So $\Ext^1_{\mathcal{C}/E}(X, \ker f)$ is nontrivial.
So by contrapositive, vanishing of $\Ext^1_{\mathcal{C}/E}(X, Y)$ for all $Y$ implies that $X$ is $E$-projective.
\end{proof}

\begin{lemma}\label{direct sum of elements of E is in E}
Suppose $\mathcal{C}$ is an abelian category, $E$ an allowable class in $\mathcal{C}$ with sectile epics.
Suppose $\mathcal{C}$ has enough $E$-projectives. Then any finite direct sum of members of $E$ is in $E$.
\end{lemma}
\begin{proof}
Let $I$ be a finite set and 
\[ 0\rightarrow X_i\rightarrow Y_i\rightarrow Z_i \rightarrow 0\]
be a member of $E$ for every $i\in I$. Then, for any $E$-projective object $P$ of $\mathcal{C}$,
we have the commutative diagram
\[\xymatrix{ 
\hom_{\mathcal{C}}(P, \oplus_i Y_i) \ar[r] \ar[d]^{\cong} & \hom_{\mathcal{C}}(P, \oplus_i Z_i) \ar[d]^{\cong} \\
\oplus_i \hom_{\mathcal{C}}(P, Y_i) \ar[r] & \oplus_i \hom_{\mathcal{C}}(P, Z_i).
}\]
The bottom horizontal map is a surjection of abelian groups, so the top horizontal map is as well. Now by Heller's theorem~\ref{heller's theorem},
the map $\oplus_i Y_i\rightarrow\oplus_i Z_i$ is an $E$-epimorphism.
So the short exact sequence
\[ 0 \rightarrow \oplus_i X_i \rightarrow \oplus_i Y_i\rightarrow \oplus_i Z_i \rightarrow 0\]
is in $E$.
\end{proof}

\begin{lemma} {\bf (Shearing $E$-monics.)} \label{shearing E-monics}
Let $\mathcal{C}$ be an abelian category and let $E$ be an allowable class in $\mathcal{C}$.
Suppose $X,Y,Z$ are objects in $\mathcal{C}$ and suppose we have $E$-monomorphisms
$e: X\rightarrow Y$ and $f: Z\rightarrow Y$.
Let $s$ be the morphism 
\[ s: X\oplus Z\rightarrow Y\oplus Z\]
given by the matrix of maps
\[ s= \left[ \begin{array}{ll} e & 0 \\ f & \id_Z \end{array}\right] .\]
Then $\coker s$ is naturally isomorphic to $\coker e$.
Furthermore, if $\mathcal{C}$ has enough $E$-injectives and $E$ has retractile monics, then $s$ is an $E$-monomorphism.
\end{lemma}
\begin{proof}
We first show that $\coker e\cong \coker s$. But this follows immediately from the commutative diagram with exact rows and exact columns:
\[\xymatrix{
0 \ar[r]\ar[d] & 0 \ar[r]\ar[d] & 0 \ar[r]\ar[d] & 0 \ar[r]\ar[d] & 0 \ar[d] \\
0 \ar[r] \ar[d] & X \ar[r]^{e} \ar[d]^i & Y \ar[r]\ar[d]^i & \coker e \ar[r]\ar[d] & 0 \ar[d] \\
0 \ar[r] \ar[d] & X\oplus Z \ar[r]^s \ar[d]^{\pi} & Y\oplus Z \ar[d]^{\pi} \ar[r] & \coker s \ar[d] \ar[r] & 0 \ar[d] \\
0 \ar[r] \ar[d] & Z \ar[r]^{\id} \ar[d] & Z \ar[d] \ar[r] & 0 \ar[d] \ar[r] & 0 \ar[d] \\
0\ar[r] & 0\ar[r] & 0\ar[r] & 0\ar[r] & 0
}\]
in which the maps marked $\pi$ are projections to the second summand, and the maps marked $i$ are inclusions as the first summand.

Now assume that $E$ has retractile monics, and let $t: Y\oplus Z \rightarrow Y\oplus Y$ be the map given by the matrix of maps
\[ t = \left[ \begin{array}{ll} \id_Y & 0 \\ -f & f \end{array} \right] .\]
Then a matrix multiplication reveals that the composite map
$t\circ s: X\oplus Z\rightarrow Y\oplus Y$ is the direct sum map $e\oplus f$, a direct sum of $E$-monomorphisms, hence by 
Lemma~\ref{direct sum of elements of E is in E}, itself an $E$-monomorphism. (Note that, by taking the opposite category and noticing that
the definition of an allowable class in an abelian category is self-dual, we get the conclusion of Lemma~\ref{direct sum of elements of E is in E}
if $E$ has retractile monics and $\mathcal{C}$ has enough $E$-injectives.)
Now since $t\circ s$ is an $E$-monomorphism and
$E$ is assumed to have retractile monics, $s$ is also an $E$-monomorphism.
\end{proof}

\begin{lemma}\label{split monos are in E}
Suppose $\mathcal{C}$ is an abelian category, $E$ an allowable class in $\mathcal{C}$ which has sectile epics.
Suppose $\mathcal{C}$ has enough $E$-projectives. 
Any split monomorphism in $\mathcal{C}$ is an $E$-monomorphism.
Furthermore, any split epimorphism in $\mathcal{C}$ is an $E$-epimorphism.
\end{lemma}
\begin{proof}
Any split monomorphism $f$ fits into a short exact sequence
\begin{equation}\label{ses in stupid lemma} 0 \rightarrow X \stackrel{f}{\longrightarrow} Y\rightarrow \coker f \rightarrow 0 \end{equation}
in which $Y$ decomposes as $X \oplus \coker f$, i.e., short exact sequence~\ref{ses in stupid lemma} is a direct sum of the short exact
sequences
\[ 0 \rightarrow X \rightarrow X \rightarrow 0 \rightarrow 0\]
and 
\[ 0 \rightarrow 0 \rightarrow \coker f\rightarrow \coker f\rightarrow 0,\]
both of which are in $E$ by the definition of an allowable class. 
Now by Lemma~\ref{direct sum of elements of E is in E}, short exact sequence~\ref{ses in stupid lemma} is in $E$.
So $f$ is an $E$-monomorphism. 

A totally analogous argument proves the same statement for split epimorphisms being $E$-epimorphisms.
\end{proof}

\begin{lemma}\label{composite of E-epics is E-epic}
Suppose $\mathcal{C}$ is an abelian category, $E$ an allowable class in $\mathcal{C}$ with sectile epics.
Suppose $\mathcal{C}$ has enough $E$-projectives.
A composite of $E$-epimorphisms is an $E$-epimorphism.
\end{lemma}
\begin{proof}
Let $f:X\rightarrow Y$ and $g: Y \rightarrow Z$ be $E$-epimorphisms. Let $P$ be an $E$-projective object equipped with a map
$P \rightarrow Z$. Then, since $g$ is an $E$-epimorphism, $P\rightarrow Z$ lifts over $g$ to a map $P \rightarrow Y$, which in turn lifts over
$f$ since $f$ is an $E$-epimorphism. So every map from an $E$-projective to $Z$ lifts over $g\circ f$. Now, by Heller's theorem~\ref{heller's theorem},
$g\circ f$ is an $E$-epimorphism.
\end{proof}

The following two lemmas are the relative-homological-algebraic generalizations of the main results of Hilton and Rees's paper \cite{MR0124377}.
We provide proofs, but they are fairly easy generalizations of those already in the literature.
\begin{lemma} \label{stable hom and nat trans of ext1}
Suppose $\mathcal{C}$ is an abelian category with an allowable class $E$.
Suppose $\mathcal{C}$ has enough $E$-projectives.
Then $E$-stable equivalence classes of morphisms $X\rightarrow Y$ in $\mathcal{C}$ are in bijection
with natural transformations of functors 
\[ \Ext^1_{\mathcal{C}/E}(Y, -)\rightarrow
\Ext^1_{\mathcal{C}/E}(X, -).\] 
This bijection is natural in $X$ and $Y$.
\end{lemma}
\begin{proof}
This proof is a straightforward generalization of Margolis' proof of the Hilton-Rees result, as in Prop.~9 of section 14.1 of \cite{MR738973}.
Write $[X,Y]$ for $\hom_{\mathcal{C}}(X, Y)$ modulo $E$-stable equivalence.
We have the morphism of abelian groups
\[ \alpha: \hom_{\mathcal{C}}(X, Y) \rightarrow {\rm nat}(\Ext^1_{\mathcal{C}/E}(Y, -) , \Ext^1_{\mathcal{C}/E}(X, -)) \]
defined by the functoriality of $\Ext^1_{\mathcal{C}/E}$ in the first variable.
If $f: X\rightarrow Y$ factors through a $E$-projective then clearly $\alpha(f) = 0$, so 
$\alpha$ factors as
\[ \xymatrix{ 
\hom_{\mathcal{C}}(X, Y) \ar[rr]^{\alpha} \ar[rd] & & {\rm nat}(\Ext^1_{\mathcal{C}/E}(Y, -) , \Ext^1_{\mathcal{C}/E}(X, -)) \\
 & [X, Y]\ar[ur]_{\beta} & }\]

We now check that $\beta$ is a bijection. Suppose $\beta(f) = 0$. Then
choose $E$-projective covers $s_X: PX \rightarrow X$ and $s_Y: PY \rightarrow Y$.
We have the commutative diagram with exact rows
\begin{equation}\label{yet another lifting diagram}
\xymatrix{
0 \ar[r] \ar[d] & \ker s_x \ar[d]_{\ker s_f} \ar[r]^{i_X} & PX\ar@{-->}[ld]_{\ell} \ar[r]^{s_X} \ar[d]^{Pf} & X\ar[d]^f \ar[r] & 0\ar[d] \\
0 \ar[r] & \ker s_Y \ar[r]_{i_Y} & PY \ar[r]_{s_Y}  & Y \ar[r] & 0 .}\end{equation}
(The map as in the dotted line has not yet been shown to exist.)
After applying $\hom_{\mathcal{C}}(-,\ker s_Y)$, we have the commutative diagram with exact rows
\[ \xymatrix{
0 & \Ext^1_{\mathcal{C}/E}(X, \ker s_Y) \ar[l] & \hom_{\mathcal{C}}(\ker s_X, \ker s_Y) \ar[l] & \hom_{\mathcal{C}}(PX, \ker s_Y) \ar[l]_{\phi} \\
0\ar[u] & \Ext^1_{\mathcal{C}/E}(Y, \ker s_Y) \ar[u]^0 \ar[l] & \hom_{\mathcal{C}}(\ker s_Y, \ker s_Y) \ar[u]^{\lambda}\ar[l] & \hom_{\mathcal{C}}(PY, \ker s_Y) \ar[l]\ar[u] . }\]
Commutativity of the diagram together with exactness of the rows and triviality of the far left-hand vertical map implies that
the map $\lambda$ factors through the image of $\phi$, i.e., there exists a map $\ell$ as in the dotted line in diagram~\ref{yet another lifting diagram}
making the triangle involving $\ker s_f, i_X,$ and $\ell$ commute. We now replace $Pf$ with $g = Pf - i_Y\circ \ell$ to get the commutative diagram with exact rows
\begin{equation}\label{yet another lifting diagram 2}
\xymatrix{
0 \ar[r] \ar[d] & \ker s_x \ar[d]^{\ker s_f} \ar[r]^{i_X} & PX \ar[r]^{s_X} \ar[d]^{g} & X\ar[d]^f \ar[r] \ar@{-->}[ld]^{\mu} & 0 \\
0 \ar[r] & \ker s_Y \ar[r]_{i_Y} & PY \ar[r]_{s_Y}  & Y \ar[r] & 0 .}\end{equation}
Since $g\circ i_X  =  Pf \circ i_X - i_Y\circ\ell\circ i_X  = 0$, 
there exists a map as in the dotted line in diagram~\ref{yet another lifting diagram 2} to make the triangle involving
$s_X, g$, and $\mu$ commute. Now we have 
\begin{eqnarray*} f \circ s_X & = & s_Y \circ g \\
 & = & s_Y \circ \mu\circ s_X \end{eqnarray*}
and $s_X$ is $E$-epic, hence epic, i.e., right-cancellable, so $f = s_Y\circ \mu$.
So $f$ factors through the $E$-projective $PY$, i.e., $f$ is $E$-stably equivalent to zero. So $\beta$ is one-to-one.

Now choose a natural transformation $\Ext^1_{\mathcal{C}/E}(Y, -)\rightarrow \Ext^1_{\mathcal{C}/E}(X, -)$.
We choose $E$-projective covers $s_X: PX\rightarrow X$ and $s_Y: PY\rightarrow Y$ as above.
Write $\chi\in\Ext^1_{\mathcal{C}/E}(Y, \ker s_Y)$ for the class of the extension 
\[ 0 \rightarrow \ker s_Y\rightarrow PY\rightarrow Y\rightarrow 0.\]
Notice that the natural map $\hom_{\mathcal{C}}(\ker s_X, \ker s_Y) \rightarrow \Ext^1_{\mathcal{C}/E}(X, \ker s_Y)$ 
is surjective since $PX$ is $E$-projective and hence $\Ext^1_{\mathcal{C}/E}(PX, \ker s_Y)\cong 0$.
So we can choose an element $h\in\hom_{\mathcal{C}}(\ker s_X, \ker s_Y)$ whose image in
$\Ext^1_{\mathcal{C}/E}(X, \ker s_Y)$ agrees with the image of $\chi$ under the given map
\[ \Ext^1_{\mathcal{C}/E}(Y, \ker s_Y)\rightarrow \Ext^1_{\mathcal{C}/E}(X, \ker s_Y).\]
The map $\Ext^1_{\mathcal{C}/E}(Y, \ker s_Y)\rightarrow \Ext^1_{\mathcal{C}/E}(Y, PY)$ is automatically zero,
so from the commutative diagram
\[ \xymatrix{ \Ext^1_{\mathcal{C}/E}(Y, \ker s_Y) \ar[r]^0 \ar[d] & \Ext^1_{\mathcal{C}/E}(Y, PY) \ar[d] \\
\Ext^1_{\mathcal{C}/E}(X, \ker s_Y) \ar[r]  & \Ext^1_{\mathcal{C}/E}(X, PY) }\]
we know that the image of $\chi$ in $\Ext^1_{\mathcal{C}/E}(X, PY)$ is zero. Hence also the image of $h$ in
$\Ext^1_{\mathcal{C}/E}(X, PY)$ is zero. Hence in the commutative diagram with exact columns
\[\xymatrix{  & \hom_{\mathcal{C}}(PX, PY) \ar[d] \\
\hom_{\mathcal{C}}(\ker s_X, \ker s_Y) \ar[r] \ar[d] & \hom_{\mathcal{C}}(\ker s_X, PY)\ar[d] \\ 
\Ext^1_{\mathcal{C}/E}(X, \ker s_Y)\ar[r] & \Ext^1_{\mathcal{C}/E}(X, PY) }\]
the image of $h$ in $\hom_{\mathcal{C}}(\ker s_X, PY)$ lifts to an element in 
$\hom_{\mathcal{C}}(PX, PY)$, i.e., we have a commutative diagram
\[\xymatrix{ 0 \ar[r] \ar[d] & \ker s_X \ar[d]^h \ar[r] & PX \ar[r]\ar[d] & X \ar[r]\ar[d] & 0 \\
0 \ar[r]  & \ker s_Y  \ar[r] & PY \ar[r] & Y \ar[r] & 0 }\]
and the map $X\rightarrow Y$ is the desired map inducing the given natural transformation in $\Ext^1_{\mathcal{C}/E}$.
Hence $\beta$ is surjective, hence an isomorphism.
\end{proof}

\begin{lemma}\label{iso in ext1 means iso after direct sum with projectives}
Suppose $\mathcal{C}$ is an abelian category with an allowable class $E$ with sectile epics.
Suppose $\mathcal{C}$ has enough $E$-projectives. Then a map 
$f: X\rightarrow Y$ in $\mathcal{C}$ induces a natural isomorphism
\begin{equation}\label{natural iso in ext1} \Ext^1_{\mathcal{C}/E}(Y, -) \rightarrow \Ext^1_{\mathcal{C}/E}(X, -)\end{equation}
if and only if there exist $E$-projective objects $P,Q$ and an isomorphism
$g: X \oplus P \rightarrow Y\oplus Q$ 
such that the composite
\[ X \stackrel{i}{\longrightarrow} X\oplus P \stackrel{g}{\longrightarrow} Y\oplus Q \stackrel{p}{\longrightarrow} Y\]
is equal to $f$. (Here we write $i$ for inclusion of the first summand and $p$ for projection to the first summand.)
\end{lemma}
\begin{proof}
If an isomorphism $g$ exists as described then we have natural isomorphisms
\begin{eqnarray*}
\Ext^1_{\mathcal{C}/E}(Y, -) & \cong & \Ext^1_{\mathcal{C}/E}(Y, -) \oplus \Ext^1_{\mathcal{C}/E}(Q ,-) \\
 & \cong & \Ext^1_{\mathcal{C}/E}(Y\oplus Q, -) \\ 
 & \cong & \Ext^1_{\mathcal{C}/E}(X\oplus P, -) \\
 & \cong & \Ext^1_{\mathcal{C}/E}(X, -) \oplus \Ext^1_{\mathcal{C}/E}(P ,-) \\
 & \cong & \Ext^1_{\mathcal{C}/E}(X, -)\end{eqnarray*}
as desired.

For the converse: suppose $f: X\rightarrow Y$ induces the natural isomorphism~\ref{natural iso in ext1}.
Then, by Lemma~\ref{stable hom and nat trans of ext1}, $f$ is an $E$-stable equivalence. So there exists a map
$g: Y \rightarrow X$ such that $\id_X - g\circ f$ and $\id_Y - f\circ g$ each factors through an $E$-projective object.
Suppose $P$ is an $E$-projective object and $i: Y \rightarrow P$ and $s: P\rightarrow Y$ maps in $\mathcal{C}$ such that
$s\circ i = \id_Y - f\circ g$. Since $\mathcal{C}$ has enough $E$-projectives, we can choose $P$ so that $s$ is epic. 
Then we have a short exact sequence in $E$
\begin{equation}\label{constructed ses in E}  0 \rightarrow \ker m \rightarrow X \oplus P \stackrel{m}{\longrightarrow} Y \rightarrow 0 \end{equation}
where $m$ is the map given by the matrix of maps
\[ m =\left[ \begin{array}{ll} f \\ s \end{array}
\right] . \]
That $m$ is an $E$-epimorphism follows from it being the composite
\[ X\oplus P \stackrel{n}{\longrightarrow} X \oplus Y \stackrel{\pi}{\longrightarrow} Y \]
where $n$ is given by the matrix of maps
\[ n = \left[ \begin{array}{ll}
\id_X & f \\
0 & s
\end{array}\right] \]
and $\pi$ is the projection to the second summand; $n$ is an $E$-epimorphism by the dual of Lemma~\ref{shearing E-monics},
$\pi$ is an $E$-epimorphism by Lemma~\ref{split monos are in E}, and $m = \pi\circ n$ is an $E$-epimorphism by
Lemma~\ref{composite of E-epics is E-epic}. 

Since $m$ is an $E$-epimorphism, by definition the short exact sequence~\ref{constructed ses in E} is in $E$. So short exact sequence~\ref{constructed ses in E}
induces a natural long exact sequence in $\Ext_{\mathcal{C}/E}$ for any object $M$ of $\mathcal{C}$:
\begin{equation}\label{yet another LES} \xymatrix{ 
\Ext_{\mathcal{C}/E}^i(X\oplus P, M)  & & \\
 \Ext_{\mathcal{C}/E}^i(Y, M) \ar[u]^{\cong} & \Ext_{\mathcal{C}/E}^{i-1}(\ker m, M)\ar[l] &
\Ext_{\mathcal{C}/E}^{i-1}(X\oplus P, M) \ar[l] \\ & &  \Ext_{\mathcal{C}/E}^{i-1}(Y, M) \ar[u]^{\cong} }\end{equation}
where the maps marked as isomorphisms are isomorphisms for $i\geq 2$ since $X\rightarrow Y$ is an $E$-stable equivalence.
Exactness of long exact sequence~\ref{yet another LES} gives us that $\Ext_{\mathcal{C}/E}^{i-1}(\ker m, M)\cong 0$ for all
$M$ in $\mathcal{C}$ and all $i\geq 2$. So by Lemma~\ref{easy lemma on nonvanishing of Ext^1}, $\ker m$ is an $E$-projective.
Part of the long exact sequence induced in $\Ext_{\mathcal{C}/E}$ by short exact sequence~\ref{constructed ses in E} reads:
\[ \xymatrix{
\Ext^1_{\mathcal{C}/E}(X \oplus P, M)    &
\Ext^1_{\mathcal{C}/E}(Y, M) \ar[l]^{\cong} &  \\
\hom_{\mathcal{C}}(\ker m, M)\ar[ur] &
\hom_{\mathcal{C}}(X \oplus P, M)\ar[l] &
\hom_{\mathcal{C}}(Y, M)\ar[l] \\ & & 0\ar[u],
}\]
i.e., 
\[ 0 \rightarrow \hom_{\mathcal{C}}(Y, M) \rightarrow \hom_{\mathcal{C}}(X\oplus P, M) \rightarrow \hom_{\mathcal{C}}(\ker m, M) \rightarrow 0\]
is exact for every object $M$ in $\mathcal{C}$. Hence short exact sequence~\ref{constructed ses in E} is in fact {\em split}, and
$X\oplus P \cong Y\oplus \ker m$, proving the lemma.
\end{proof}

\begin{lemma} \label{shearing iso}
{\bf (Shearing isomorphism.)} Suppose $X,Y$ are objects in an abelian category $\mathcal{C}$ and $f: X \rightarrow Y$ is a monomorphism.
Then the pushout $Y\coprod_X Y$ is naturally isomorphic to $Y\oplus \coker f$.
\end{lemma}
\begin{proof}
Let $g: Y\oplus Y \rightarrow Y\oplus Y$ be the map given by the (invertible, hence an isomorphism) matrix of maps
\[ g = \left[ \begin{array}{ll}  \id & -\id \\ 0 & \id \end{array} \right] .\]
Then we have the commutative diagram with exact rows
\[\xymatrix{
0 \ar[r]\ar[d] & X\ar[r]^{[ f\ f ]} \ar[d]^{\id} & Y\oplus Y\ar[r]\ar[d]^{g} & Y\coprod_X Y \ar[d] \ar[r] & 0 \ar[d] \\
0 \ar[r] & X \ar[r]^{[f\ 0 ]} & Y\oplus Y \ar[r] & \coker f \oplus Y \ar[r] & 0}\]
and hence the isomorphism $Y\coprod_X Y\stackrel{\cong}{\longrightarrow} \coker f\oplus Y$.
\end{proof}

\begin{lemma}\label{existence of objects of E-proj dim 1}
Suppose $\mathcal{C}$ is an abelian category, $E$ an allowable class in $\mathcal{C}$. 
Suppose $\mathcal{C}$ has enough $E$-projectives. If there exists an object of finite $E$-projective dimension $n\geq 2$ in $\mathcal{C}$,
then there exists an object of $E$-projective dimension $1$ in $\mathcal{C}$.\end{lemma}
\begin{proof}
Let $X$ have $E$-projective dimension $n\geq 2$. Choose an $E$-projective resolution of $X$
\[ 0 \rightarrow P_n \stackrel{f_n}{\longrightarrow} P_{n-1} \stackrel{f_{n-1}}{\longrightarrow} \dots \stackrel{f_2}{\longrightarrow} 
P_1 \stackrel{f_1}{\longrightarrow} P_{0} \stackrel{}{\longrightarrow} X \rightarrow 0.\]
Then 
\begin{equation} \label{E-proj res} 0 \rightarrow P_n \stackrel{f_n}{\longrightarrow} P_{n-1} \stackrel{f_{n-1}}{\longrightarrow} \dots  \stackrel{f_2}{\longrightarrow}
P_2 \stackrel{}{\longrightarrow} \coker f_2  \rightarrow 0\end{equation}
is an $E$-projective resolution of $\coker f_2$. If $\coker f_2$ is $E$-projective, then 
\[ 0\rightarrow \coker f_2 \rightarrow P_0\rightarrow X\rightarrow 0\]
is a length $1$ $E$-projective resolution of $X$, and we are done. So suppose $\coker f_2$ is not $E$-projective. So it is not of $E$-projective dimension zero.
Then diagram~\ref{E-proj res} expresses $\coker f_2$ as having $E$-projective dimension at most $n-1$. 
Now we continue by induction: either $\coker f_3$ is $E$-projective or has $E$-projective dimension at most $n-2$, etc. 
After at most $n$ steps this process terminates with an object of $E$-projective dimension $1$.
\end{proof}

The preceding lemmas suffice for us to prove the following proposition, which is really a negative result: it shows that, if $\mathcal{D}$-indexed
colimits of $E$-projectives are not always $E$-projective, then $\mathcal{D}$-indexed homotopy colimits fail to be unique up to homotopy.
\begin{prop}\label{existence of htpy colimits implies projectivity of colimits of projectives}
Let $\mathcal{D}$ be a small category and let $\mathcal{C}$ be a weak Waldhausen abelian category.  
Choose a class $E$ allowable for $\mathcal{C}$.
Suppose $\mathcal{D}$-indexed homotopy colimits are well-defined in $\mathcal{C}$. Then every $\mathcal{D}$-indexed colimit of $E$-projective objects of $\mathcal{C}$ is
$E$-projective.
\end{prop}
\begin{proof}
We work by contrapositive. Suppose there exists a homotopy colimit diagram $F:\mathcal{D}\rightarrow \mathcal{C}$ with the property that
$F(X)$ is $E$-projective for every object $X$ of $\mathcal{D}$, and $\colim F$ is not $E$-projective.
Let $G: \mathcal{D}\rightarrow \mathcal{C}$ be the zero diagram, i.e., $G(X) = 0$ for all objects $X$ of $\mathcal{D}$.
Then the unique map $\phi: F\rightarrow G$ has the property that $\phi(X)$ is an $E$-stable equivalence for every object $X$ of $\mathcal{D}$,
since any map with $E$-projective domain and $E$-projective codomain is trivially an $E$-stable equivalence.
But $\colim F$ is nontrivial, so $\Ext^1_{\mathcal{C}/E}(\colim F, Y)$ is nontrivial for some object $Y$ of $\mathcal{C}$, 
by Lemma~\ref{easy lemma on nonvanishing of Ext^1}.
So the natural transformation $0 \cong \Ext^1_{\mathcal{C}/E}(\colim G, Y) \rightarrow \Ext^1_{\mathcal{C}/E}(\colim F, Y)$ is not
an isomorphism. So by Lemma~\ref{stable hom and nat trans of ext1}, the map $\colim F\rightarrow \colim G \cong 0$ is not
an $E$-stable equivalence.
\end{proof}

\begin{corollary} \label{main nonexistence corollary on htpy pushouts}
Let $\mathcal{C}$ be a weak Waldhausen abelian category. Let $E$ be a class allowable for $\mathcal{C}$.
Suppose $\mathcal{C}$ has enough $E$-projectives. 
Then homotopy pushouts, and in particular homotopy cofibers, are well-defined in $\mathcal{C}$ only if every object in $\mathcal{C}$
has $E$-projective dimension $0$ or $\infty$.
\end{corollary}
\begin{proof}
Suppose an object in $\mathcal{C}$ has finite $E$-projective dimension $n>0$. Then, by Lemma~\ref{existence of objects of E-proj dim 1}, there exists
some object $X$ in $\mathcal{C}$ of $E$-projective dimension $1$. 
Choose an $E$-projective resolution 
\[ 0\rightarrow P_1 \stackrel{s}{\longrightarrow} P_0\rightarrow X\rightarrow 0\]
of $X$. Then we have the commutative diagram
\[ \xymatrix{
 P_1 \ar[d]^{s} & P_1\ar[l]^{\id} \ar[d]^{\id}\ar[r]^{\id} & P_1\ar[d]^{s} \\
 P_0  & P_1\ar[l]^{s} \ar[r]^{s} & P_0  }\]
in which all vertical arrows are $E$-stable equivalences (since any map between two $E$-projective objects is an $E$-stable equivalence!)
and all horizontal arrows are $E$-monomorphisms.
We compute the induced map on pushouts by the commutative diagram with exact rows 
\[ \xymatrix{  0 \ar[r] & P_1\ar[r]^{\Delta}\ar[d]^{\id} & P_1\oplus P_1\ar[r]\ar[d]^{s \oplus s} &  P_1 \ar[d] \ar[r]  & 0 \\
 0 \ar[r] & P_1\ar[r]^{} & P_0\oplus P_0\ar[r] & P_0\coprod_{P_1} P_0  \ar[r]  & 0 .}\]
Exactness of the top row, as well as an isomorphism
$P_0\coprod_{P_1}P_0\cong X\oplus P_0$, both follow from Lemma~\ref{shearing iso}.
So the pushout map $P_1\coprod_{P_1} P_1\rightarrow P_0\coprod_{P_1}P_0$
is, up to isomorphism, the map $P_1\rightarrow X\oplus P_0$. Applying $\Ext^1_{\mathcal{C}/E}(-, M)$ to this map, we get
\[ \xymatrix{ 
\Ext^1_{\mathcal{C}/E}( P_0\coprod_{P_1}P_0 , M) \ar[r]\ar[d]^{\cong} & \Ext^1_{\mathcal{C}/E}( P_1\coprod_{P_1}P_1 , M) \ar[d]^{\cong} \\
\Ext^1_{\mathcal{C}/E}( X\oplus P_0 , M) \ar[r]\ar[d]^{\cong} & \Ext^1_{\mathcal{C}/E}( P_1 , M) \ar[d]^{\cong} \\
\Ext^1_{\mathcal{C}/E}( X , M) \ar[r] & 0 }\]
for all objects $M$ of $\mathcal{C}$. Since $X$ is assumed to be of $E$-projective dimension $1$, it is not $E$-projective, so by 
Lemma~\ref{easy lemma on nonvanishing of Ext^1},
$\Ext^1_{\mathcal{C}/E}( X , M)$ is nonzero for some object $M$ in $\mathcal{C}$.
So by Lemma~\ref{stable hom and nat trans of ext1}, the pushout map $P_1\coprod_{P_1} P_1\rightarrow P_0\coprod_{P_1}P_0$
is not an $E$-stable equivalence.
\end{proof}

\begin{corollary}\label{nonexistence result on weq2}
Let $\mathcal{C}$ be a weak Waldhausen abelian category. Let $E$ be a class allowable for $\mathcal{C}$.
Suppose $\mathcal{C}$ has enough $E$-projectives. 
Then $\mathcal{C}$ is not a Waldhausen category unless 
each object in $\mathcal{C}$
has $E$-projective dimension either $0$ or $\infty$.\end{corollary}
\begin{proof}
Well-definedness of homotopy colimits is Waldhausen's axiom {\bf Weq 2} in the 
definition of a Waldhausen category, from \cite{MR802796}.
\end{proof}

\section{Positive results on homotopy pushouts and homotopy cofibers.}

In this section we prove that homotopy pushouts in a weak Waldhausen abelian category $\mathcal{C}$ are well-defined if one makes some mild assumptions on 
$\mathcal{C}$, as well as one quite significant assumption on $\mathcal{C}$: that every relatively projective object is relatively injective. In the absolute case,
i.e., the case where $E$ is the class of all short exact sequences in $\mathcal{C}$, this condition is somewhat weaker than the assumption that
$\mathcal{C}$ be quasi-Frobenius, which holds when $\mathcal{C}$ is the category of modules over any quasi-Frobenius ring 
(e.g. connected co-commutative finite-dimensional Hopf algebras over fields, such as finite-dimensional sub-Hopf-algebras of the Steenrod algebra). Recall that an abelian category
$\mathcal{C}$ is said to be ``quasi-Frobenius'' if projective and injective objects coincide in $\mathcal{C}$.

The following proposition is the relative-homological-algebraic generalization of the main result of Oort's paper \cite{MR0146233}. The proof follows Oort's
but with some adaptations (which are easy) to the relative situation.
\begin{prop}\label{generalization of oort's result}
Let $\mathcal{C}$ be an abelian category, $E$ an allowable class in $\mathcal{C}$ with sectile epics. Suppose $\mathcal{C}$ has enough $E$-projectives.
Then a map $f:X\rightarrow Y$ in $\mathcal{C}$ induces a natural isomorphism
\[ \Ext^i_{\mathcal{C}/E}(Y, -) \rightarrow \Ext^i_{\mathcal{C}/E}(X, -)\]
if and only if there exists a short exact sequence in $E$
\begin{equation} \label{oort ses} 0 \rightarrow Q \rightarrow P\oplus X \stackrel{p}{\longrightarrow} Y \rightarrow 0\end{equation}
with $P$ an $E$-projective and $Q$ of $E$-projective dimension $\leq i-1$,
such that the composite $X\stackrel{i}{\longrightarrow} P\oplus X \stackrel{p}{\longrightarrow} Y$ is equal to $f$.
(Here $i$ is inclusion as the second summand.)
\end{prop}
\begin{proof}
If short exact sequence~\ref{oort ses} exists with the described properties, then for any object $M$ of $\mathcal{C}$, the induced long exact sequence
in $\Ext_{\mathcal{C}/E}$ reads
\[ \xymatrix{ 
0\cong\Ext^i_{\mathcal{C}/E}(Q, M) & \Ext^i_{\mathcal{C}/E}(P\oplus X, M) \ar[l] & \Ext^i_{\mathcal{C}/E}(Y, M) \ar[l] & \Ext^{i-1}_{\mathcal{C}/E}(Q, M) \cong 0\ar[l] ,
}\]
and since $P$ is projective we now have 
\[ \Ext^i_{\mathcal{C}/E}(Y, M) \cong \Ext^i_{\mathcal{C}/E}(P\oplus X, M)\cong \Ext^i_{\mathcal{C}/E}(X, M).\]

Now instead assume that $f$ induces an isomorphism in $\Ext^i_{\mathcal{C}/E}$. We want to construct a short exact sequence~\ref{oort ses} with the described
properties. Choose exact sequences for $X,Y$
\begin{equation}\xymatrix{
0 \ar[r]\ar[d] & N \ar[d]^{\theta_{i-1}} \ar[r]^{d^Q_{i-1}} & Q_{i-2} \ar[d]^{\theta_{i-2}}\ar[r]^{d^Q_{i-2}} & \dots \ar[r]^{d^Q_2}  & Q_1 \ar[r]^{d^Q_1}\ar[d]^{\theta_1} & Q_0 \ar[r]^{d^Q_0}\ar[d]^{\theta_0} & X \ar[d]^f \ar[r] & 0 \ar[d] \\
0 \ar[r] & M  \ar[r]^{d^P_{i-1}} & P_{i-2} \ar[r]^{d^P_{i-2}} & \dots \ar[r]^{d^P_2}  & P_1 \ar[r]^{d^P_1} & P_0^{d^P_0} \ar[r] & Y  \ar[r] & 0  
}\end{equation}
with each $Q_j$ and $P_j$ an $E$-projective, with each short exact sequence
\[ 0 \rightarrow \im d^Q_{j+1} \rightarrow Q_{j} \rightarrow \im d^Q_j \rightarrow 0\]
and
\[ 0 \rightarrow \im d^P_{j+1} \rightarrow P_{j} \rightarrow \im d^P_j \rightarrow 0\]
both in $E$, and 
such that each $\theta_j$ factors through the kernel of the map $Q_{j-1}\rightarrow Q_{j-2}$ for $j\geq 1$, and $\theta_0$ factors through the kernel
of the map $Q_0\rightarrow X$. Then $\theta_{n-1}$ induces a natural isomorphism
\[ \Ext^1_{\mathcal{C}/E}(N, -) \cong \Ext^i_{\mathcal{C}/E}(X, -) \cong \Ext^i_{\mathcal{C}/E}(Y, -) \cong \Ext^1_{\mathcal{C}/E}(M, -) \]
and hence, by Lemma~\ref{iso in ext1 means iso after direct sum with projectives}, there exist $E$-projectives $Q,P$ and an isomorphism $N\oplus Q\cong M\oplus P$ extending $\theta_{n-1}$.
We take a direct sum with these $E$-projectives to get the commutative diagram with exact rows
\begin{equation}\xymatrix{
0 \rightarrow N\oplus Q \ar[d]^{\cong} \ar[r]^{} & Q_{i-2}\oplus Q \ar[d]^{}\ar[r]^{} & 
Q_{i-3} \ar[r] \ar[d] & Q_{i-4} \ar[r]\ar[d] & \dots \\
0 \rightarrow M\oplus P  \ar[r]^{} & P_{i-2}\oplus P \ar[r]^{} &
P_{i-3} \ar[r] & P_{i-4} \ar[r] & \dots \\
}\end{equation}
and, regarding this as a double complex and totalizing, we get an exact sequence
\[ 0 \rightarrow Q_{i-2}\oplus Q \rightarrow Q_{i-3}\oplus P_{i-2}\oplus P \rightarrow Q_{i-2}\oplus P_{i-3}\rightarrow Q_{i-4}\oplus P_{i-3}\rightarrow \dots
\rightarrow X\oplus P_0\stackrel{p}{\longrightarrow} Y \rightarrow 0\]
which expresses that we have a short exact sequence
\begin{equation}\label{putative oort ses} 0 \rightarrow \ker p \rightarrow X\oplus P_0 \stackrel{p}{\longrightarrow} Y\rightarrow 0\end{equation}
with $P_0$ $E$-projective and $\ker p$ of $E$-projective dimension $\leq i-1$, and with $p$ extending $f$ as desired.
All that remains is to check that short exact sequence~\ref{putative oort ses} is in $E$. This follows from $p$ being a difference of the map
$f: X\rightarrow Y$ and the $E$-projective cover $d^P_{0}: P_0 \rightarrow Y$, as follows: by the dual of Lemma~\ref{shearing E-monics},
the map $X\oplus P_0 \rightarrow X \oplus Y$ given by the matrix of maps
\[ \left[ \begin{array}{ll}
\id_X & f \\
0 & p 
\end{array} \right]\]
is an $E$-epimorphism. By Lemma~\ref{split monos are in E}, the projection $X\oplus Y\rightarrow Y$ is an $E$-epimorphism as well.
So the composite map $X\oplus P_0 \rightarrow X\oplus Y \rightarrow Y$, which is equal to $p$, is an $E$-epimorphism by Lemma~\ref{composite of E-epics is E-epic}.
\end{proof}

\begin{lemma}\label{ext2 equiv implies ext1 equiv}
Let $\mathcal{C}$ be an abelian category and let $E$ be an allowable class in $\mathcal{C}$ with sectile epics.
Suppose $\mathcal{C}$ has enough $E$-projectives, and suppose that every $E$-projective object is $E$-injective. 
If $X\rightarrow Y$ is a map in $\mathcal{C}$ which induces a natural isomorphism of functors
\begin{equation}\label{natural iso in ext2} \Ext^2_{\mathcal{C}/E}(Y, -) \stackrel{\cong}{\longrightarrow} \Ext^2_{\mathcal{C}/E}(X, -),\end{equation}
then $X\rightarrow Y$ is an $E$-stable equivalence.
\end{lemma}
\begin{proof}
Let $f: X\rightarrow Y$ be a map which induces a natural isomorphism of functors~\ref{natural iso in ext2}. Then, by 
Prop.~\ref{generalization of oort's result}, there exists a short exact sequence in $E$
\begin{equation}\label{ses to split} 0 \rightarrow Q \rightarrow P\oplus X \rightarrow Y \rightarrow 0\end{equation}
where $P\oplus X \rightarrow Y$ extends $f$ and both $P$ and $Q$ are $E$-projective. Now we use the assumption that
every $E$-projective in $\mathcal{C}$ is $E$-injective: since $Q$ is $E$-injective, short exact sequence~\ref{ses to split} splits,
and we get an isomorphism $P\oplus X \stackrel{\cong}{\longrightarrow} Q\oplus Y$. Now Lemma~\ref{iso in ext1 means iso after direct sum with projectives}
implies that $f$ induces a natural isomorphism in $\Ext^1_{\mathcal{C}/E}$ and hence, by Lemma~\ref{stable hom and nat trans of ext1},
$f$ is an $E$-stable equivalence.
\end{proof}

\begin{lemma}\label{left properness}
Let $\mathcal{C}$ be a weak Waldhausen abelian category.  Let $E$ be a class allowable for $\mathcal{C}$.
Suppose $\mathcal{C}$ has enough $E$-projectives and enough $E$-injectives, 
suppose $E$ has sectile epics and retractile monics, and
suppose every $E$-projective object is $E$-injective. 
Then, for any cofibration $f: X\rightarrow Y$ and any weak equivalence
$g: X \rightarrow Z$ in $\mathcal{C}$, the pushout map $Y \rightarrow Y\coprod_X Z$ is also a weak equivalence.
\end{lemma}
\begin{proof}
The given maps fit into a commutative diagram with exact rows in $E$
\[ \xymatrix{
0\ar[r]\ar[d] & X \ar[r]^f \ar[d]^g\ & Y \ar[r]\ar[d] & \coker f \ar[r]\ar[d]^{\cong} & 0 \ar[d] \\
0 \ar[r] & Z \ar[r] & Y\coprod_X Z \ar[r] & \coker f\ar[r] & 0 
}\]
and since $g$ is an $E$-stable equivalence, Lemma~\ref{stable hom and nat trans of ext1} gives us, for any object $M$ of $\mathcal{C}$, 
the marked isomorphisms in the commutative diagram
with exact columns
\[\xymatrix{
\Ext^1_{\mathcal{C}/E}(Z, M)\ar[d]\ar[r]^{\cong} &  \Ext^1_{\mathcal{C}/E}(X, M)\ar[d] \\
\Ext^2_{\mathcal{C}/E}(\coker f, M)\ar[d]\ar[r]^{\cong} & \Ext^2_{\mathcal{C}/E}(\coker f, M)\ar[d] \\
\Ext^2_{\mathcal{C}/E}(Y\coprod_X Z, M)\ar[d]\ar[r] & \Ext^2_{\mathcal{C}/E}(Y, M)\ar[d] \\
\Ext^2_{\mathcal{C}/E}(Z, M)\ar[d]\ar[r]^{\cong} & \Ext^2_{\mathcal{C}/E}(X, M)\ar[d] \\
\Ext^3_{\mathcal{C}/E}(\coker f, M)\ar[r]^{\cong} & \Ext^3_{\mathcal{C}/E}(\coker f, M) .}\]
So by the Five Lemma, $Y \rightarrow Y\coprod_X Z$ induces a natural isomorphism of functors
\[ \Ext^2_{\mathcal{C}/E}(Y\coprod_X Z, -) \stackrel{\cong}{\longrightarrow} \Ext^2_{\mathcal{C}/E}(Y, -) \]
and hence, by Lemma~\ref{ext2 equiv implies ext1 equiv}, $Y\rightarrow Y\coprod_X Z$ is an $E$-stable equivalence, hence a weak equivalence.
\end{proof}

\begin{prop}\label{main prop on well-definedness of homotopy pushouts}
Let $\mathcal{C}$ be a weak Waldhausen abelian category.  Let $E$ be a class allowable for $\mathcal{C}$.  
Suppose $\mathcal{C}$ has enough $E$-projectives and enough $E$-injectives, 
suppose $E$ has sectile epics and retractile monics,
and suppose that every $E$-projective object is $E$-injective. Then homotopy pushouts are well-defined in $\mathcal{C}$.
\end{prop}
\begin{proof}
Let $\mathcal{D}$ be the small category indexing pushout diagrams, 
i.e., $\mathcal{D}$ has three objects $A_0,A_1,A_2$, maps $A_0\rightarrow A_1$ and $A_0\rightarrow A_2$,
and no other non-identity maps. Suppose $F,G: \mathcal{D}\rightarrow\mathcal{C}$ are homotopy colimit diagrams and
$\phi: F\rightarrow G$ is a natural transformation such that $\phi(X): F(X) \rightarrow G(X)$ is an $E$-stable equivalence
for every object $X$ of $\mathcal{D}$. Then we have the commutative diagram with rows short exact sequences in $E$:
\[ \xymatrix{
0 \ar[r]\ar[d] & F(A_0) \ar[d]\ar[r] & F(A_1)\oplus F(A_2)\ar[d]\ar[r] & \colim F \ar[r]\ar[d] & 0\ar[d] \\
0 \ar[r] & G(A_0) \ar[r] & G(A_1)\oplus G(A_2) \ar[r] & \colim G \ar[r] & 0. }\]
That the maps $F(A_0) \rightarrow F(A_1)\oplus F(A_2)$ and $G(A_0) \rightarrow G(A_1)\oplus G(A_2)$ are 
$E$-monomorphisms, and hence that the rows are exact and in $E$, follows from Lemma~\ref{shearing E-monics}.
Now, for any object $M$ of $\mathcal{C}$, we get the commutative diagram with exact columns
\[\xymatrix{
\Ext^1_{\mathcal{C}/E}(G(A_1)\oplus G(A_2), M)\ar[d] \ar[r]^{\cong} & \Ext^1_{\mathcal{C}/E}(F(A_1)\oplus F(A_2), M)\ar[d] \\
\Ext^1_{\mathcal{C}/E}(G(A_0), M) \ar[d]\ar[r]^{\cong}  & \Ext^1_{\mathcal{C}/E}(F(A_0), M) \ar[d] \\
\Ext^2_{\mathcal{C}/E}(\coker G, M) \ar[d]\ar[r] & \Ext^2_{\mathcal{C}/E}(\coker F, M) \ar[d] \\
\Ext^2_{\mathcal{C}/E}(G(A_1)\oplus G(A_2), M)\ar[r]^{\cong} \ar[d] & \Ext^2_{\mathcal{C}/E}(F(A_1)\oplus F(A_2), M) \ar[d] \\
\Ext^2_{\mathcal{C}/E}(G(A_0), M)\ar[r]^{\cong} & \Ext^2_{\mathcal{C}/E}(F(A_0), M) }\]
where the horizontal maps marked as isomorphisms are isomorphisms by Lemma~\ref{stable hom and nat trans of ext1}. 
By the Five Lemma, the remaining horizontal map is an isomorphism.
So we have a natural isomorphism of functors 
\[ \Ext^2_{\mathcal{C}/E}(\colim G, -) \stackrel{\cong}{\longrightarrow} \Ext^2_{\mathcal{C}/E}(\colim F, -) \]
and now, by Lemma~\ref{ext2 equiv implies ext1 equiv}, the
map $\colim F \rightarrow \colim G$ is an $E$-stable equivalence. Hence homotopy pushouts are well-defined in $\mathcal{C}$.
\end{proof}

\begin{corollary}\label{existence result on weq2}
Let $\mathcal{C}$ be a weak Waldhausen abelian category.  Let $E$ be a class allowable for $\mathcal{C}$.  
Suppose $\mathcal{C}$ has enough $E$-projectives and enough $E$-injectives, 
suppose $E$ has sectile epics and retractile monics,
and suppose that every $E$-projective object is $E$-injective. 
Then $\mathcal{C}$ satisfies Waldhausen's axiom {\bf Weq 2}.
\end{corollary}
\begin{proof}
Follows immediately from
Lemma~\ref{left properness} and Prop.~\ref{main prop on well-definedness of homotopy pushouts}.
\end{proof}

\section{Positive results on sequential homotopy colimits.}

In this section we show that, under the same conditions from the previous section (that relatively projective objects are relatively injective),
sequential homotopy colimits are well-defined (Prop.~\ref{sequential colimits are well-defined}).
As a corollary, in the next section we will be able to show that geometric realization of simplicial objects in $\mathcal{C}$ is well-defined (Cor.~\ref{geometric realization is well-defined}).

\begin{lemma}\label{nice factorizations in E-projective covers}
Suppose $\mathcal{C}$ is an abelian category and $E$ is an allowable class in $\mathcal{C}$ with retractile monics. Suppose $\mathcal{C}$ has
enough $E$-projectives and enough $E$-injectives. Suppose we have objects $P,X,Y$
in $\mathcal{C}$ with $P$ an $E$-projective, and suppose we have a map $f: X\rightarrow Y$ and an $E$-epimorphism $p: P \rightarrow X$.
Then there exists an $E$-projective object $Q$, a split monomorphism $g: P \rightarrow Q$ with $E$-projective cokernel, and an $E$-epimorphism $q: Q\rightarrow Y$ making the diagram
\begin{equation}\label{factorization diagram, desired} \xymatrix{
P \ar[r]^p \ar[d]^g & X \ar[d]^f \\
Q \ar[r]^q & Y }\end{equation}
commute.
\end{lemma}
\begin{proof}
Choose an $E$-projective $P_0$ and an $E$-epimorphism $s: P_0 \rightarrow Y$. Since $P$ is $E$-projective, the composite map $f\circ p: P\rightarrow Y$
lifts over $s$ to give a map $\ell: P \rightarrow P_0$, i.e., $s\circ\ell = f\circ p$. 
Let $g: P \rightarrow P \oplus P_0$ be the map given by the matrix of maps
\[ g = \left[ \begin{array}{ll} \id_P & \ell \end{array} \right] \]
and let $q: P \oplus P_0 \rightarrow Y$ be the map given by the matrix of maps
\[ q = \left[ \begin{array}{ll} 0 \\ s \end{array} \right] . \]
It is trivial to check that the diagram
\begin{equation}\label{factorization diagram, constructed} \xymatrix{
P \ar[r]^p \ar[d]^g & X \ar[d]^f \\
P\oplus P_0 \ar[r]^q & Y }\end{equation}
commutes. The map $q$ is an $E$-epimorphism since $s$ is, by the dual of Lemma~\ref{split monos are in E} 
combined with Lemma~\ref{direct sum of elements of E is in E}. 
We also have that $\pi\circ g = \id_P$, where $\pi: P\oplus P_0 \rightarrow P$ is projection to the first summand.
So $g$ is a split monomorphism. That its cokernel is $E$-projective follows from an easy application of the Nine Lemma to get the commutative diagram
with exact rows and exact columns:
\[\xymatrix{
 & 0 \ar[r]\ar[d] & 0 \ar[r]\ar[d] & 0 \ar[d] &  \\
0 \ar[r]\ar[d] & 0 \ar[r]\ar[d] & P_0 \ar[r]^{\id} \ar[d]^i & P_0 \ar[r]\ar[d]^{\id} & 0 \ar[d] \\
0 \ar[r]\ar[d] & P \ar[r]^g \ar[d]^{\id} & P\oplus P_0 \ar[r]^{\pi^{\prime}}\ar[d]^{\pi} & P_0 \ar[r] \ar[d] & 0 \ar[d] \\
0 \ar[r] & P \ar[r]^{\id} \ar[d] & P \ar[r]\ar[d] & 0 \ar[r]\ar[d] & 0 \\
 & 0 \ar[r] & 0 \ar[r] & 0 & 
}\]
where $i$ is inclusion as the second summand and $\pi^{\prime}$ is projection to the second summand. So $P_0$, an $E$-projective, is the cokernel of $g$.

So diagram~\ref{factorization diagram, constructed} is the desired diagram~\ref{factorization diagram, desired}.
\end{proof}

For the next proposition, we use the phrase ``sequential colimit'' to describe any colimit indexed by the partially-ordered set of the
natural numbers regarded as a category, i.e., a colimit with shape
\[ \bullet \rightarrow \bullet \rightarrow \bullet \rightarrow \dots .\]
We also include, in the next proposition, a requirement that sequential colimits of $E$-long exact sequences be $E$-long exact.
This is a mild assumption; in the absolute case, when $E$ is the class of all short exact sequences in $\mathcal{C}$, this assumption
is equivalent to Grothendieck's axiom AB5.
We remind the reader that axiom AB5 on an abelian category $\mathcal{C}$ stipulates that 
small colimits exist in $\mathcal{C}$ and that a sequential colimit of exact sequences in $\mathcal{C}$ remains exact. This axiom is satisfied,
for example, by the category of $R$-modules, for any ring $R$.
\begin{prop}\label{sequential colimits are well-defined}
Let $\mathcal{C}$ be a weak Waldhausen abelian category. Let $E$ be a class allowable for $\mathcal{C}$, and suppose that a sequential colimit
of $E$-long exact sequences in $\mathcal{C}$ is $E$-long exact.
Suppose $\mathcal{C}$ has enough $E$-projectives and enough $E$-injectives, 
suppose $E$ has retractile monics and sectile epics,
and suppose that every $E$-projective object is $E$-injective.
Then sequential homotopy colimits are well-defined in $\mathcal{C}$.
\end{prop}
\begin{proof}
Let $F, G: \mathbb{N}\rightarrow \mathcal{C}$ be homotopy colimit diagrams and let $\phi: F\rightarrow G$ be a natural transformation (i.e., map of diagrams)
such that $\phi(n): F(n)\rightarrow G(n)$ is an $E$-stable equivalence for every $n\in \mathbb{N}$. In other words, we have a commutative diagram
\begin{equation}
\xymatrix{
F(0) \ar[r]\ar[d] & F(1) \ar[r]\ar[d] & F(2) \ar[r]\ar[d] & \dots \\
G(0) \ar[r] & G(1) \ar[r] & G(2) \ar[r] & \dots  }
\end{equation}
in which all horizontal maps are cofibrations (in particular, $E$-monomorphisms) and all vertical maps are $E$-stable equivalences.
Then, by Lemma~\ref{stable hom and nat trans of ext1}, the vertical maps in the diagram
\[\xymatrix{
\Ext^n_{\mathcal{C}/E}(F(0), M) & \Ext^n_{\mathcal{C}/E}(F(1), M) \ar[l] & \Ext^n_{\mathcal{C}/E}(F(2), M) \ar[l] & \dots \ar[l] \\
\Ext^n_{\mathcal{C}/E}(G(0), M) \ar[u] & \Ext^n_{\mathcal{C}/E}(G(1), M) \ar[l] \ar[u] & \Ext^n_{\mathcal{C}/E}(G(2), M) \ar[l] \ar[u] & \dots \ar[l] }\]
are isomorphisms for all $n\geq 1$ and for any object $M$ of $\mathcal{C}$, and these isomorphisms are natural in $M$. Consequently, we have 
natural isomorphisms
\begin{equation}\label{natural iso in lim} \lim_i \Ext^n_{\mathcal{C}/E}(G(i), M) \stackrel{\cong}{\longrightarrow} \lim_i \Ext^n_{\mathcal{C}/E}(F(i), M) \end{equation}
and
\begin{equation}\label{natural iso in lim1} R^1\lim_i \Ext^n_{\mathcal{C}/E}(G(i), M) \stackrel{\cong}{\longrightarrow} R^1\lim_i \Ext^n_{\mathcal{C}/E}(F(i), M) \end{equation}
for all $n\geq 1$, where we are writing $R^1\lim$ for the first right-derived functor of $\lim$.

Now we can use Lemma~\ref{nice factorizations in E-projective covers} to choose $E$-projective resolutions for each $F(i)$ to get a commutative diagram
\[\xymatrix{ 
\dots \ar[r] & P^F_{2,0} \ar[r]\ar[d] & P^F_{1,0} \ar[r]\ar[d] & P^F_{0,0}\ar[r]\ar[d] & F(0) \ar[r]\ar[d] & 0\ar[d] \\
\dots \ar[r] & P^F_{2,1} \ar[r]\ar[d] & P^F_{1,1} \ar[r]\ar[d] & P^F_{0,1}\ar[r]\ar[d] & F(1) \ar[r]\ar[d] & 0\ar[d] \\
\dots \ar[r] & P^F_{2,2} \ar[r]\ar[d] & P^F_{1,2} \ar[r]\ar[d] & P^F_{0,2}\ar[r]\ar[d] & F(2) \ar[r]\ar[d] & 0\ar[d] \\
       & \vdots  & \vdots & \vdots & \vdots & \vdots }\]
in which each row is an $E$-projective resolution and each vertical map $P^F_{i,j}\rightarrow P^F_{i,j+1}$ is a split monomorphism.

We claim that, for every $i\in \mathbb{N}$, the colimit $\colim_j P^F_{i,j}$ is $E$-projective. Let $f: X\rightarrow Y$ be an $E$-epimorphism
and $g: \colim_j P^F_{i,j}\rightarrow Y$ be a map. We write $g_j: P^F_{i,j}\rightarrow Y$ for the $j$th component map of $g$. Since
$P^F_{i,0}$ is $E$-projective, there exists a lift $\ell_0: P^F_{i,0}\rightarrow X$ of $g_0$ over $f$, which provides the first step of an induction. Suppose
we have a map $\ell_j: P^F_{i,j} \rightarrow X$ such that $f \circ \ell_j = g_j$. Then, since the map $P^F_{i,j}\rightarrow P^F_{i,j+1}$ is a split monomorphism
with $E$-projective cokernel, there exists a map $\ell_{j+1}: P^F_{i,j+1} \rightarrow X$ making the diagram
\[ \xymatrix{ F(i) \ar[d]^{\ell_i} \ar[r] & F(i+1) \ar[d]^{g_{i+1}} \ar@{-->}[dl]^{\ell_{i+1}} \\ X \ar[r]_f & Y }\]
commute. So we can assemble the maps $\{ \ell_j\}_{j\in \mathbb{N}}$ into a map $\ell: \colim_j P^F_{i,j}\rightarrow X$ such that
$f\circ \ell = g$. So $\colim_jP^F_{i,j}$ has the universal property defining an $E$-projective object, so $\colim_j P^F_{i,j}$ is $E$-projective.
We make the same constructions for $G$ as well as $F$, writing $P^G_{i,j}$ rather than $P^F_{i,j}$ for the $E$-projectives constructed in this way.

Now the chain complex
\begin{equation}\label{colimit resolution} 
\dots \rightarrow \colim_j P^F_{2,j} \rightarrow \colim_j P^F_{1,j} \rightarrow \colim_j P^F_{0,j} \rightarrow \colim F \rightarrow 0\end{equation}
is $E$-long exact, due to our assumption that sequential colimits of $E$-long exact sequences in $\mathcal{C}$ are $E$-long exact.
Furthermore, we have shown that each $\colim_j P^F_{i,j}$ is $E$-projective. So long exact sequence~\ref{colimit resolution} is an $E$-projective resolution of
$\colim F$. So, for any object $M$ of $\mathcal{C}$, the cohomology of the cochain complex
\begin{equation}
\xymatrix{
\dots   & \hom_{\mathcal{C}}(\colim_j P^F_{1,j},M)\ar[l]\ar[d]^{\cong} & \hom_{\mathcal{C}}(\colim_j P^F_{0,j},M)\ar[l] \ar[d]^{\cong} & 0\ar[l] \ar[d] \\
\dots & \lim_j \hom_{\mathcal{C}}( P^F_{1,j},M)\ar[l] & \lim_j \hom_{\mathcal{C}}( P^F_{0,j},M)\ar[l]  & 0\ar[l] 
}\end{equation}
is $\Ext_{\mathcal{C}/E}^*(\colim F, M)$, and this isomorphism is natural in $M$. 

Now we have the usual short exact sequence relating the cohomology of a sequential limit of cochain complexes of abelian groups to the sequential
limit of their cohomologies:
\[ 0 \rightarrow R^1\lim_j H^{n-1}\hom_{\mathcal{C}}(P^F_{\bullet, j}, M) \rightarrow
 H^n \lim_j\hom_{\mathcal{C}}(P^F_{\bullet, j}, M) \rightarrow \lim_j H^n \hom_{\mathcal{C}}(P^F_{\bullet, j}, M) \rightarrow 0 .\]
Due to isomorphisms~\ref{natural iso in lim} and \ref{natural iso in lim1}, we now have the commutative diagram with exact rows:
\[\xymatrix{
0 \ar[r]\ar[d] & R^1\lim_j \Ext^{n-1}_{\mathcal{C}/E}( G(j), M) \ar[r]\ar[d]^{\cong} &
 \Ext^n_{\mathcal{C}/E}(\colim G, M) \ar[r]\ar[d]^{\cong} &  \lim_j \Ext^n_{\mathcal{C}/E}(G(j), M) \ar[r]\ar[d]^{\cong} & 0 \ar[d] \\
0 \ar[r]\ar[d] & R^1\lim_j H^{n-1}\hom_{\mathcal{C}}(P^G_{\bullet, j}, M) \ar[r]\ar[d]^{b} &
 H^n \lim_j\hom_{\mathcal{C}}(P^G_{\bullet, j}, M) \ar[r]\ar[d]^{a} &  \lim_j H^n \hom_{\mathcal{C}}(P^G_{\bullet, j}, M) \ar[r]\ar[d]^{c} & 0 \ar[d] \\
0 \ar[r]\ar[d] & R^1\lim_j H^{n-1}\hom_{\mathcal{C}}(P^F_{\bullet, j}, M) \ar[r]\ar[d]^{\cong} &
 H^n \lim_j\hom_{\mathcal{C}}(P^F_{\bullet, j}, M) \ar[r]\ar[d]^{\cong} &  \lim_j H^n \hom_{\mathcal{C}}(P^F_{\bullet, j}, M) \ar[r]\ar[d]^{\cong} & 0 \ar[d] \\
0 \ar[r] & R^1\lim_j \Ext^{n-1}_{\mathcal{C}/E}( F(j), M) \ar[r] &
 \Ext^n_{\mathcal{C}/E}(\colim F, M) \ar[r] &  \lim_j \Ext^n_{\mathcal{C}/E}(F(j), M) \ar[r] & 0 }\]
and the vertical map marked $b$ is an isomorphism if $n\geq 2$, and the vertical map marked $c$ is an isomorphism if $n\geq 1$.
Hence, by the Five Lemma, the vertical map marked $c$ is an isomorphism if $n\geq 2$.

We conclude that the map
\[  \Ext^n_{\mathcal{C}/E}(\colim G, M)\rightarrow  \Ext^n_{\mathcal{C}/E}(\colim F, M)\]
is an isomorphism if $n\geq 2$. Hence, by Lemma~\ref{ext2 equiv implies ext1 equiv}, the assumption that every $E$-projective is 
$E$-injective implies that the map $\colim F\rightarrow \colim G$ is an $E$-stable equivalence.
So sequential homotopy colimits in $\mathcal{C}$ are unique up to homotopy.
\end{proof}

\section{Positive results on geometric realization.}

In this section we prove that, under the same assumptions made in the previous section plus the assumption that our abelian category satisfies Grothendieck's
axiom AB3, geometric realization of simplicial objects is well-defined (Cor.~\ref{geometric realization is well-defined}).

One can approach geometric realization as a particular
kind of colimit called a ``co-end''; this approach is taken in e.g. \cite{MR1712872}. We take a different approach to geometric realization in this section,
by regarding geometric realization as the sequential colimit of a certain
sequence of homotopy cofibers. We give an abbreviated description of this approach in Def.~\ref{def of gr tower},
but it is well-known in the special case of a pointed (e.g. stable) model category, and, for example, 
it appears in the context of a triangulated category in \cite{MR1214458}.

\begin{lemma}\label{pullback of surjections is surjective}
A pullback of a surjective map of abelian groups is surjective.
\end{lemma}
\begin{proof}
The forgetful functor from abelian groups to sets is a right adjoint, hence preserves limits.
It also clearly preserves surjections.
So the lemma is true if a pullback of a surjective maps of sets is surjective, which is 
an elementary exercise.
\end{proof}

\begin{lemma} 
\label{when admissible monics are closed under pushout}
Let $\mathcal{C}$ be an abelian category and let
$E$ be an allowable class with retractile
monics. Then $E$-monics
are closed under pushout in $\mathcal{C}$. That is,
if $X\rightarrow Z$ is an $E$-monic and $X\rightarrow Y$ is any morphism
in $\mathcal{C}$, then the canonical
map $Y\rightarrow Y\coprod_X Z$ is
an $E$-monic.
\end{lemma}
\begin{proof}
Suppose $f: X\rightarrow Z$ is an $E$-monic and $X\rightarrow Y$ any morphism. 
We have the commutative diagram with exact rows
\[ \xymatrix{
0 \ar[r] & 
 X \ar[r]^f \ar[d] & 
 Z \ar[r]\ar[d] &
 \coker f \ar[r] \ar[d] &
 0 \ar[d] \\
 &
Y \ar[r] &
Y\coprod_X Z \ar[r] &
\coker f \ar[r] & 
 0 }\]
and hence, for every $E$-injective $I$, 
the induced commutative diagram
of abelian groups
\[ \xymatrix{
0 \ar[r] \ar[d] & 
 \hom_{\mathcal{C}}(\coker f, I) \ar[r] &
 \hom_{\mathcal{C}}(Z, I) \ar[r] &
 \hom_{\mathcal{C}}(X, I) \ar[r] &
 0 \\
0 \ar[r]  & 
 \hom_{\mathcal{C}}(\coker f, I) \ar[r]\ar[u]^{\cong} &
 \hom_{\mathcal{C}}(Y\coprod_X Z, I) \ar[r]\ar[u] &
 \hom_{\mathcal{C}}(Y, I). \ar[u] & }\]
Exactness of the top row follows from $f$
being an $E$-monic together with
$E$ having retractile monics, hence
$E$ is its own retractile closure, hence 
$E$-monics are precisely the
maps which induce a surjection after applying
$\hom_{\mathcal{C}}(- ,I)$ for every $E$-injective $I$.
Now in particular we have a commutative square
in the above commutative diagram:
\[ 
\xymatrix{
 \hom_{\mathcal{C}}(Z, I) \ar[r] &
 \hom_{\mathcal{C}}(X, I)  \\
 \hom_{\mathcal{C}}(Y\coprod_X Z, I) \ar[r]\ar[u] &
 \hom_{\mathcal{C}}(Y, I), \ar[u]
}\]
which is a pullback square of abelian groups,
by the universal property of the pushout.
The top map in the square is a surjection,
hence so is the bottom map, Lemma~\ref{pullback of surjections is surjective}.
So $\hom_{\mathcal{C}}(Y\coprod_X Z, I)
\rightarrow \hom_{\mathcal{C}}(Y, I)$
is a surjection for every $E$-injective $I$. Again since
$E$ is its own retractile closure, this implies
that $Y\rightarrow Y\coprod_X Z$ is an $E$-monic.
\end{proof}

We now define a weak form of the quasi-Frobenius condition that will allow us to factor maps into cofibrations followed by weak equivalences:
\begin{definition}\label{def of cone-frob}
Suppose $\mathcal{C}$ is a weak Waldhausen abelian category and $E$ a class allowable for $\mathcal{C}$. 
We say that $\mathcal{C}$ is {\em cone-Frobenius} if there exists a functor $J: \mathcal{C}\rightarrow\mathcal{C}$
and a natural transformation $\eta: \id_{\mathcal{C}}\rightarrow J$ such that:
\begin{enumerate}
  \item $J(X)$ is $E$-projective for every object $X$ of $\mathcal{C}$,
  \item $\eta(X): X\rightarrow J(X)$ is an $E$-monomorphism for every object $X$ of $\mathcal{C}$, and
  \item if $f: X\rightarrow Y$ is an $E$-monomorphism then so is $J(f): J(X)\rightarrow J(Y)$.
\end{enumerate}
We sometimes call the pair $J,\eta$ a {\em cone functor on $\mathcal{C}$.} 
\end{definition}
Here is an example of a cone functor: suppose $R$ is a Noetherian ring, and let $U$ be the injective envelope of the direct sum $\bigoplus R/I$, where $I$ ranges
across all right ideals of $R$. For each right $R$-module $M$, let 
\[ J(M) = \prod_{\hom_R(M, U)} U,\]
and let $\eta(M): M \rightarrow J(M)$ send $m$ to the map whose component in the factor corresponding to $f\in\hom_R(M, U)$ is $f(m)$.
This gives a functorial embedding of every $R$-module into an injective $R$-module, and (due to the characterizing property of an injective envelope)
$J$ sends monomorphisms to monomorphisms (this construction is due to Bass, and appears in Ex.~5.26 in Lam's book \cite{MR2278849}). If every injective $R$-module is projective, then $J$ is a cone functor. So the category of modules over any
Noetherian quasi-Frobenius ring is cone-Frobenius, for example. 
Similarly, the category of finitely-generated modules over any finite-dimensional quasi-Frobenius algebra over a finite field is also cone-Frobenius.

Now we will begin assuming our weak Waldhausen abelian category $\mathcal{C}$ is cone-Frobenius. 
The notation we will usually use is this: 
for any object $X$ of $\mathcal{C}$, we will write $i_X: X\rightarrow P_X$ for the chosen $E$-monomorphism from $X$ to an $E$-projective $P_X$.

The following lemma gives us conditions under which a homotopy pushout can be computed as the pushout of a diagram in which only {\em one} map is a cofibration,
rather than both maps. 
\begin{lemma}\label{simple pushouts}
Let $\mathcal{C}$ be a cone-Frobenius weak Waldhausen abelian category.  Let $E$ be a class allowable for $\mathcal{C}$.
Suppose $\mathcal{C}$ has enough $E$-projectives and enough $E$-injectives, 
suppose $E$ has sectile epics and retractile monics, 
suppose every $E$-projective object is $E$-injective.
Then any map $f: X\rightarrow Y$
can be factored as a composite
\begin{equation}\label{desired factorization} X \stackrel{f_0}{\longrightarrow} \tilde{Y} \stackrel{f_1}{\longrightarrow} Y,\end{equation}
where $f_0$ is a cofibration and $f_1$ is a weak equivalence, and furthermore,
the pushout of the diagram
\begin{equation}\label{misc pushout diag 1}\xymatrix{ X \ar[r]^{f_0} \ar[d]^{i_X} & \tilde{Y} \\ P_X & }\end{equation}
is weakly equivalent to the pushout of the diagram
\begin{equation}\label{misc pushout diag 2} \xymatrix{ X \ar[r]^{f} \ar[d]^{i_X} & Y \\ P_X . & }\end{equation}
\end{lemma}
\begin{proof}
We let $\tilde{Y}$ be $P_X \oplus Y$, we let $f_0$ be given by the matrix of maps
\[ f_0 = \left[ \begin{array}{ll} i_X & f \end{array} \right] ,\]
and we let $f_1$ be the projection to the second summand. It is trivial to check that $f = f_1\circ f_0$.
That $f_1$ is a weak equivalence follows from its being a split epimorphism with $E$-projective kernel.
That $f_0$ is a cofibration follows from its being the composite of the inclusion in the first summand $X \rightarrow X\oplus Y$ 
(which is an $E$-monomorphism, hence cofibration, by Lemma~\ref{split monos are in E})
followed by the map  $X\oplus Y\rightarrow P_X\oplus Y$ given by the matrix of maps
\[ \left[ \begin{array}{ll} i_x & f \\ 0 & \id_Y \end{array} \right], \]
which is an $E$-monomorphism, hence cofibration, by Lemma~\ref{shearing E-monics}.
So $f_0$ is a composite of two cofibrations, hence itself a cofibration. So we have the desired factorization~\ref{desired factorization}.

Now we have the commutative diagram
\begin{equation}\label{yet another comm diagram}\xymatrix{ 
X\ar[d]^{i_X} \ar[r]^{f_0} & P_X\oplus Y\ar[d] \ar[r]^{f_1} & Y \ar[d] \\
P_X \ar[r] & P_X\coprod_{X}(P_X\oplus Y) \ar[r] & (P_X\coprod_X (P_X\oplus Y)) \coprod_{P_X\oplus Y} Y 
}\end{equation}
in which the two squares are pushout squares, hence the outer rectangle is a pushout diagram, i.e.,
we have a natural isomorphism
\[ \left(P_X\coprod_X (P_X\oplus Y)\right) \coprod_{P_X\oplus Y} Y \cong P_X\coprod_X Y.\]
In diagram~\ref{yet another comm diagram}, the maps $i_X$ and $f_0$ are cofibrations, hence the central vertical map is as well, by Lemma~\ref{when admissible monics are closed under pushout}.
Hence the bottom map
\[ P_X\coprod_X (P_X\oplus Y)\rightarrow P_X\coprod_X Y\]
is a pushout of a weak equivalence ($f_1$) along a cofibration (the central vertical map), hence itself a weak equivalence by Lemma~\ref{left properness}.
So we have a weak equivalence, as desired, between the pushout of the diagram~\ref{misc pushout diag 1} and the pushout of diagram~\ref{misc pushout diag 2}.
\end{proof}

Recall that an abelian category satisfies Grothendieck's axiom AB3 if it has arbitrary (small) colimits. Since any abelian category has coequalizers, axiom
AB3 is equivalent to having arbitrary (small) coproducts.
\begin{definition}\label{def of gr tower}
Suppose $\mathcal{C}$ is a cone-Frobenius weak Waldhausen abelian 
category satisfying Grothendieck's axiom AB3 as well as the assumptions of Lemma~\ref{simple pushouts}.
Let $F: \Delta^{\op}\rightarrow\mathcal{C}$ be
a simplicial object in $\mathcal{C}$.
We will write $F_n$ for the $n$th object of $F$ and
\[ d_i: F_n \rightarrow F_{n-1},\]
$i = 0, ... , n$, for the face maps of $F$.
We write $\Sigma X$ for the pushout of the diagram
\[\xymatrix{ X \ar[r]^{i_X} \ar[d]^{i_x} & P_X \\ P_X . & }\]

Then by a 
{\em geometric realization tower of $F$} we mean the diagram
$GR_F: \mathbb{N}\rightarrow \mathcal{C}$ defined inductively as follows:
\[ GR_F(0) = F_0\oplus \coprod_{n>0} P_{F_n},\]
and if $GR_F(i)$ has already been defined for $i=0, \dots , n-1$, we define $GR_F(n)$ as the pushout in the diagram
\begin{equation}\label{pushout defining GR} \xymatrix{ \Sigma^{n-1} (F_n\oplus\coprod_{i>n}P_{F_i}) \ar[d]^{\Sigma^{n-1}(i_{F_n}\oplus \id)} \ar[r]^{f_n} & GR_F(n-1)\ar[d]^{g_n} \\
\Sigma^{n-1}\left(\coprod_{i\geq n} P_{F_i} \right) \ar[r] & GR_F(n) }\end{equation}
where $f_n: \Sigma^{n-1} (F_n\oplus \coprod_{i > n} P_{F_i})  \rightarrow GR_F(n-1)$ is the map obtained from
the two nulhomotopies (i.e., factorizations through an $E$-projective) of the composite map
\[ \Sigma^{n-2} (F_n\oplus \coprod_{i>n}P_{F_i})  \stackrel{\Sigma^{n-2}d}{\longrightarrow} 
\Sigma^{n-2} (F_{n-1}\oplus \coprod_{i>n-1} P_{F_i}) \stackrel{f_{n-1}}{\longrightarrow}
 GR_F(n-2) \stackrel{g_{n-1}}{\longrightarrow} GR_F(n-1) , \]
where we write \[ d: F_n\oplus \coprod_{i>n}P_{F_i}\rightarrow F_{n-1}\oplus \coprod_{i>n-1}P_{F_i}\]  
for the map given on the summands $F_n, P_{F_{n+1}}, P_{F_{n+2}}, \dots$ of its domain as follows:
\begin{itemize}
\item
on the summand $F_n$ of its domain, it is the alternating sum \[ d_0 - d_1 + d_2 - \dots + (-1)^{n-1} d_{n-1}: F_n\rightarrow F_{n-1}\]
plus the map $i_{F_n}: F_n \rightarrow P_{F_n}$,
\item
and on each summand $P_{F_i}$ for $i>n$, it is simply the inclusion of the summand $P_{F_i}$ into the codomain.
\end{itemize}

Finally, by the {\em geometric realization of $F$} we mean the colimit $\colim GR_F$ of the geometric realization tower of $F$.
\end{definition}
Here is a simple way of describing the maps $d: F_n\oplus \coprod_{i>n}P_{F_i}\rightarrow F_{n-1}\oplus \coprod_{i>n-1}P_{F_i}$ appearing
in Def.~\ref{def of gr tower}: if one applies the factorization~\ref{desired factorization} from Lemma~\ref{simple pushouts}
to the alternating sum map $d_0 - d_1 + d_2 - \dots + (-1)^{n-1} d_{n-1}: F_n\rightarrow F_{n-1}$,
one gets a cofibration $f_0: F_n \rightarrow F_{n-1}\oplus P_{F_n}$. The map $d$ is simply the direct sum of $f_0$ with the identity map on all 
the $E$-projectives $P_{F_i}$ that will appear later on in the geometric realization tower.
Note that, if $P$ is $E$-projective, then so is $\Sigma P$, since the assumptions made in Def.~\ref{def of gr tower} imply
that homotopy pushouts are well-defined in $\mathcal{C}$ due to Prop.~\ref{main prop on well-definedness of homotopy pushouts}, which
in turn implies that a homotopy pushout of $E$-projective objects is $E$-projective, by 
Prop.~\ref{existence of htpy colimits implies projectivity of colimits of projectives}. Furthermore, the cone-Frobenius assumption
implies that $i_{F_n}$ being a cofibration forces $\Sigma i_{F_n}$ to be a cofibration.
So the left-hand vertical map in square~\ref{pushout defining GR}
really is an $E$-monomorphism into an $E$-projective object. So by Lemma~\ref{simple pushouts}, square~\ref{pushout defining GR} is computing a homotopy pushout.

\begin{lemma}
Under the assumptions made in Def.~\ref{def of gr tower}, making different choices of the cone functor, in particular the 
$E$-projective objects $P_X$ and cofibrations $i_X$, does not change the $E$-stable equivalence type
of each object $GR_F(n)$ in the geometric realization tower.
\end{lemma}
\begin{proof}
Suppose we have two choices $P_X,P_X^{\prime}$ of $E$-projective object and two choices of cofibration, 
$i_X: X\rightarrow P_X$ and $i_X^{\prime}: X\rightarrow P_X^{\prime}$. For any cofibration $f: X\rightarrow Y$, we have the two pushout diagrams
\[ \xymatrix{ X \ar[r]^{i_X} \ar[d]^{f} & P_X \\ Y } \mbox{\ and\ } \xymatrix{ X \ar[r]^{i_X^{\prime}} \ar[d]^{f} & P_X^{\prime} \\ Y } .\]
By Lemma~\ref{simple pushouts}, the pushout of each diagram is $E$-stably equivalent to the pushout of the diagram 
\[  \xymatrix{ X \ar[r]^{} \ar[d]^{f} & 0 \\ Y .} \]
So the $E$-stable equivalence type of the homotopy cofiber of a cofibration doesn't depend on the choice of $i_x,P_X$.

Now we handle the dependence of the factorization~\ref{desired factorization} on the choices of $i_X,P_X$. 
We have the two pushout diagrams
\begin{equation}\label{child pushout diagrams} \xymatrix{ X \ar[r]^{i_X} \ar[d]^{f_0} & P_X \\ P_X\oplus Y } \mbox{\ and\ } \xymatrix{ X \ar[r]^{i_X} \ar[d]^{f_0^{\prime}} & P_X \\ P_X^{\prime}\oplus Y } \end{equation}
arising from factorization~\ref{desired factorization}, and we want to know that the pushouts of these two diagrams
are $E$-stably equivalent. We accomplish this with the pushout diagram
\begin{equation}\label{spanning pushout diagram} \xymatrix{ X \ar[r]^{i_X} \ar[d]^{m} & P_X \\ P_X\oplus P_X^{\prime}\oplus Y } \end{equation}
where $m$ is given the matrix of maps
\[ m = \left[ \begin{array}{lll} i_X & i_X^{\prime} & f \end{array} \right] ,\]
and is a cofibration by Lemma~\ref{shearing E-monics}.
We note that diagram~\ref{spanning pushout diagram} maps to each of the two diagrams~\ref{child pushout diagrams}; we handle the map to the left-hand 
diagram~\ref{child pushout diagrams}, and the map to the right-hand diagram is handled similarly.
Since $i_X, f_0,$ and $m$ are all cofibrations, Lemma~\ref{when admissible monics are closed under pushout} implies that, in the pushout square
\[ \xymatrix{ P_X\oplus P_X^{\prime} \oplus Y \ar[d]\ar[r] & P_X\coprod_X (P_X\oplus P_X^{\prime} \oplus Y) \ar[d] \\
P_X\oplus Y \ar[r] & P_X\coprod_X (P_X\oplus Y) ,}\]
the top horizontal map is a cofibration. The left-hand vertical map is a weak equivalence, since it is a split epimorphism with $E$-projective kernel; 
so by Lemma~\ref{left properness}, the right-hand vertical map is also a weak equivalence.
A similar argument holds for the right-hand pushout diagram~\ref{child pushout diagrams}, so since the pushouts of the two 
pushout diagrams in \ref{child pushout diagrams} are each $E$-stably equivalent to the pushout of diagram~\ref{spanning pushout diagram}, 
the two pushout diagrams in \ref{child pushout diagrams} are $E$-stably equivalent to one another. So the choice of $i_X,P_X$ used in the construction of the 
factorization~\ref{desired factorization} doesn't affect the $E$-stable equivalence type of the resulting homotopy pushouts.

Finally, since $GR_F$ is constructed entirely from these two operations (factorizations as in~\ref{desired factorization} and homotopy cofibers of cofibrations),
up to levelwise $E$-stable equivalence, $GR_F$ does not depend on the choices of $i_X,P_X$.
\end{proof}

\begin{prop}\label{levelwise weak eq induces levelwise weak eq of gr towers}
Let $\mathcal{C}$ be a cone-Frobenius weak Waldhausen abelian category.  Let $E$ be a class allowable for $\mathcal{C}$.
Suppose $\mathcal{C}$ has enough $E$-projectives and enough $E$-injectives, 
suppose $E$ has sectile epics and retractile monics, 
suppose every $E$-projective object is $E$-injective.
Then, for any simplicial object $F: \Delta^{\op}\rightarrow\mathcal{C}$ in $\mathcal{C}$, the geometric realization tower
$GR_F: \mathbb{N}\rightarrow\mathcal{C}$ is a homotopy colimit diagram.
Furthermore, if $F,G: \Delta^{\op}\rightarrow\mathcal{C}$ are two simplicial objects and $\phi: F \rightarrow G$ a levelwise weak equivalence,
then the induced natural transformation $GR_{\phi}: GR_F \rightarrow GR_G$ has the property that, for every natural number $n$, the map
$GR_{\phi}(n): GR_F(n)\rightarrow GR_G(n)$ is a weak equivalence.
\end{prop}
\begin{proof}
First we check that $GR_F$ is a homotopy pushout diagram. All we need to check is that, for every natural number $n$, the map
$GR_F(n)\rightarrow GR_F(n+1)$ is a cofibration. But this map is the pushout of diagram~\ref{pushout defining GR},
which is the pushout of a cofibration along a cofibration, hence itself a cofibration by Lemma~\ref{when admissible monics are closed under pushout}.

Now suppose $F,G,\phi$ are as in the statement of the proposition. The maps $GR_{\phi}(n): GR_F(n)\rightarrow GR_G(n)$ are, by construction,
pushouts of weak equivalences along cofibrations, hence by Lemma~\ref{left properness}, they are themselves weak equivalences.
\end{proof}

\begin{corollary}\label{geometric realization is well-defined}
Let $\mathcal{C},E$ be as in Def.~\ref{def of gr tower}, and suppose that a sequential colimit
of $E$-long exact sequences in $\mathcal{C}$ is $E$-long exact.
Then geometric realization of simplicial objects in $\mathcal{C}$ is well defined. That is, if we have two simplicial objects $F,G: \Delta^{\op}\rightarrow \mathcal{C}$ in $\mathcal{C}$ and a natural transformation $\phi:F \rightarrow G$ such that $\phi(n): F(n)\rightarrow G(n)$ is a weak equivalence for every natural 
number $n$, then the induced map of geometric realizations $\colim GR_F \rightarrow \colim GR_G$ is a weak equivalence.\end{corollary}
\begin{proof}
By Prop.~\ref{levelwise weak eq induces levelwise weak eq of gr towers}, $\phi$ induces a natural transformation $GR_{\phi}: GR_F\rightarrow GR_G$ of 
homotopy colimit diagrams which is a levelwise weak equivalence. Then by Prop.~\ref{sequential colimits are well-defined}, $GR_{\phi}$ induces
a weak equivalence of colimits $\colim GR_F \rightarrow \colim GR_G$.
\end{proof}

\section{Appendix on basic notions of relative homological algebra.}

Here is an appendix on some ideas in relative homological algebra. Since this subject is a little bit obscure, some readers might find the definitions (which are
all classical, except for Def.~\ref{def of retractile monics}) in this appendix helpful. 

\begin{definition} \label{def of allowable class}
An {\em allowable class in $\mathcal{C}$} consists of a collection 
$E$ of short exact sequences in $\mathcal{C}$ which is closed under isomorphism
of short exact sequences and which contains every short exact sequence
in which at least one object is the zero object of $\mathcal{C}$.
(See section IX.4 of \cite{MR1344215} for this definition
and basic properties.)
\end{definition}
The usual ``absolute'' homological algebra in an abelian category 
$\mathcal{C}$ is recovered by letting the allowable class $E$ consist
of {\em all} short exact sequences in $\mathcal{C}$.

Once one chooses an allowable class $E$, one has the notion of monomorphisms
relative to $E$, or ``$E$-monomorphisms,'' and epimorphisms relative to $E$,
or ``$E$-epimorphisms.''
\begin{definition}
Let $E$ be an allowable class in $\mathcal{C}$.
A monomorphism $f: M\rightarrow N$ 
in $\mathcal{C}$ is called an {\em $E$-monomorphism} or an {\em $E$-monic}
if the short exact sequence
\[ 0\rightarrow M\stackrel{f}{\longrightarrow} N \rightarrow \coker f\rightarrow 0\]
is in $E$.

Dually, an epimorphism $g: M\rightarrow N$ is called an 
{\em $E$-epimorphism} or an {\em $E$-epic} if the short exact sequence
\[ 0\rightarrow \ker f\rightarrow M\stackrel{f}{\longrightarrow} N \rightarrow 0\]
is in $E$.\end{definition}
In the absolute case, the case that $E$ is all short exact sequences in $\mathcal{C}$,
the $E$-monomorphisms are simply the monomorphisms, and the $E$-epimorphisms
are simply the epimorphisms.

Projective and injective objects are at the heart of homological algebra. In 
relative homological algebra, one has the notion of relative projectives, or
$E$-projectives: these are simply the objects which lift over every 
$E$-epimorphism. The $E$-injectives are defined dually.
\begin{definition}
Let $E$ be an allowable class in $\mathcal{C}$.
An object $X$ of $\mathcal{C}$ is said to be an {\em $E$-projective}
if, for every diagram
\[ \xymatrix{ & X \ar[d]  \\
 M \ar[r]^f & N }\]
in which $f$ is an $E$-epic,
there exists a morphism $X\rightarrow M$ making the above diagram commute.

Dually, 
an object $X$ of $\mathcal{C}$ is said to be an {\em $E$-injective}
if, for every diagram
\[ \xymatrix{ M\ar[r]^f\ar[d] & N   \\
X & }\]
in which $f$ is an $E$-monic,
there exists a morphism $N\rightarrow X$ making the above diagram commute.

When the allowable class $E$ is clear from context we sometimes refer to $E$-projectives
and $E$-injectives as {\em relative projectives} and {\em relative injectives,} respectively.
\end{definition}
In the absolute case, the case that $E$ is all short exact sequences 
in $\mathcal{C}$, the $E$-projectives are simply the projectives, and 
the $E$-injectives are simply the injectives.

Once one has a notion of relative projectives, one has a reasonable notion of
a stable equivalence or, loosely, a ``homotopy'' between maps, as studied
(usually in the absolute case, where $E$-projectives are simply projectives) in stable representation theory:
\begin{definition}
Let $E$ be an allowable class in $\mathcal{C}$.
Let $f,g: M\rightarrow N$ be morphisms in $\mathcal{C}$.
We say that $f$ and $g$ are {\em $E$-stably equivalent}
and we write $f\simeq g$
if $f-g$ factors through an $E$-projective object of 
$\mathcal{C}$. 
\end{definition}

One then has the notion of stable equivalence of objects, or loosely,
``homotopy equivalence'':
\begin{definition}\label{def of stable equivalence}
We say that a map $f: M\rightarrow N$ is a
{\em $E$-stable equivalence} if there exists a map
$h: N\rightarrow M$ such that
$f\circ h\simeq \id_N$ and $h\circ f\simeq \id_M$.
\end{definition}
In the absolute case where $E$ consists of all short exact
sequences in $\mathcal{C}$, this is the usual notion of stable
equivalence of modules over a ring. 
Over a Hopf algebra over a field,
stably equivalent modules have isomorphic cohomology in 
positive degrees, so if one is serious about computing
the cohomology of all finitely-generated modules over a
particular Hopf algebra---such as the Steenrod algebra or
the group ring of a Morava stabilizer group---it 
is natural to first compute the 
representation ring modulo stable equivalence.

Here is a new definition which makes many arguments involving allowable classes substantially smoother:
\begin{definition} \label{def of retractile monics}
An allowable class $E$ is said to 
{\em have retractile monics}
if, whenever $g\circ f$ is an $E$-monic, $f$ is also an $E$-monic.

Dually, an allowable class $E$ is said to 
{\em have sectile epics}
if, whenever $g\circ f$ is an $E$-epic, $g$ is also an $E$-epic.
\end{definition}

The utility of the notion of ``having sectile epics'' comes from the following fundamental theorem of
relative homological algebra, due to Heller (see \cite{MR1344215}), whose statement is slightly cleaner is one is willing to use the phrase ``having sectile epics.'' The consequence
of Heller's theorem is that, in order to specify a ``reasonable'' allowable class
in an abelian category, it suffices to specify the relative projective objects
associated to it.
\begin{theorem} {\bf (Heller.)}\label{heller's theorem}
If $\mathcal{C}$ is an abelian category and $E$ is an allowable class in $\mathcal{C}$ with sectile epics and enough $E$-projectives,
then an epimorphism $M\rightarrow N$ in $\mathcal{C}$ is an $E$-epic if and
only if the induced map $\hom_{\mathcal{C}}(P, M)\rightarrow 
\hom_{\mathcal{C}}(P, N)$ of abelian groups is an epimorphism for all
$E$-projectives $P$.
\end{theorem}

\bibliography{/home/asalch/texmf/tex/salch}{}
\bibliographystyle{plain}
\end{document}